\documentclass[preprint,3p]{elsarticle}
\usepackage{marvosym}
\usepackage{amsthm}
\usepackage{times}
\usepackage{graphicx}
\usepackage{latexsym}
\usepackage{todonotes}
\usepackage{comment}
\usepackage{enumerate}
\usepackage{tipa}

\newtheorem*{notation}{Notation}
\newtheorem{question}{Question}
\usepackage{amsmath,capt-of}

\usepackage{mathrsfs}
\usepackage{makeidx}
\usepackage{multicol}
\usepackage[bottom]{footmisc}

\usepackage{listings}
\usepackage{enumerate}
\usepackage{url}
\usepackage{tikz}
\usetikzlibrary{shapes.geometric, arrows}
\usepackage{booktabs} 
\usepackage{empheq}

\newcommand{\ignore}[1]{}

\usetikzlibrary{shapes.arrows}
\usetikzlibrary{decorations.pathreplacing}

\newcommand{\structure}[1]{\ensuremath{\left(#1\right)}\xspace}
\newcommand{\set}[1]{\ensuremath{\left\{#1\right\}}\xspace}
\newcommand{\Z}{\ensuremath{\mathbb{Z}}\xspace}
\usepackage{complexity}
\newclass{\EnumP}{EnumP\xspace}

\newlang{\CMP}{CMP\xspace }
\newlang{\SOBFID}{SOBFID\xspace}
\newlang{\ESOBFID}{EnumSOBFID\xspace}
\newcommand{\ttnode}{{\textit{tt}}\xspace}
\newcommand{\rootnode}{{\textit{root}}\xspace}

\usepackage[ruled,vlined]{algorithm2e}

\newcommand{\DDS}{\textsc{DDS}\xspace}
\newcommand{\ie}{\emph{i.e.}\@\xspace}
\newcommand{\etc}{\emph{etc.}\@\xspace}
\newcommand{\wrt}{\emph{w.r.t.}\@\xspace}
\usepackage{algpseudocode}

\usepackage{mathtools}

\definecolor{bleu}{rgb}{0, 0.6, 0.8}
\definecolor{rose}{rgb}{0.8, 0, 0.4}
\definecolor{vert}{rgb}{0, 0.6, 0.4}

\def\rose#1{\textcolor{rose}{#1}}

 \def\sr#1{\rose{\textbf{SR : #1}}}

\tikzstyle{condition}  = [diamond, draw, fill=bleu!50, text width=5em, text badly centered, inner sep=0pt]
\tikzstyle{conditionFunction}  = [diamond, draw, fill=white, text width=5em, text badly centered, inner sep=0pt]
\tikzstyle{solver}     = [draw, ellipse,fill=white, text centered,  minimum height=2em,text width=3.7em]
\tikzstyle{block}      = [rectangle, draw, fill=bleu!80, text width=7em, text centered, rounded corners, minimum height=2em]
\tikzstyle{bigblock}   = [rectangle, draw, fill=bleu!80, text width=8em, text centered, rounded corners, minimum height=2em]
\tikzstyle{line}       = [draw, -latex']
\tikzstyle{problem}    = [circle,double, draw, fill=black!50, text width=3em, text centered, rounded corners, minimum height=1em]
\tikzstyle{emptyblock} = [rectangle, fill=white, text centered, text width=2em, rounded corners, minimum height=2em]
\tikzstyle{emptybigblock} = [rectangle, fill=white, text width=4em, rounded corners, minimum height=2em]
\tikzstyle{vertex}=[circle, draw]
\usetikzlibrary{shapes.multipart}
\usetikzlibrary{shapes,arrows,oodgraph}

\newtheorem{definition}{Definition}
\newtheorem{lemma}{Lemma}
\newtheorem{theorem}{Theorem}
\newtheorem{proposition}{Proposition}
\newtheorem{corollary}{Corollary}
\newtheorem{remark}{Remark}
\newtheorem{example}{Example}
\newcommand{\N}{\mathbb{N}\xspace}
\DeclareMathOperator{\lcm}{lcm}
\usetikzlibrary{decorations.markings}
\usepackage{tabularx}

\usepackage[charter,cal=cmcal]{mathdesign}
\DeclareMathAlphabet{\mathpzc}{OT1}{pzc}{m}{it}
\usepackage{pstricks,pst-xkey,pst-asr,graphicx}\psset{everyasr=\tiershortcuts}
\usepackage{tipa}

\newcommand{\coef}{a}

\newcommand{\rterm}{b}

\newcommand{\lf}{{\ell}ab}
\newcommand{\lfun}[1]{{\ell}ab({#1})}

\newcommand{\mon}{m}

\newcommand{\nP}{l}

\newcommand{\Pright}{l_\rterm}

\newcommand{\mons}{{\text{\fontfamily{cmvtt}\itshape m}}}

\newcommand{\setst}{\mathcal{X}}

\newcommand{\nodedds}{v}

\newcommand{\dds}{\mathcal{S}}

\newcommand{\pp}{\mathcal{P}}

\newcommand{\ring}{R}

\newcommand{\qtvar}{\mathcal{V}}

\newcommand{\nodemdd}{\alpha}
\newcommand{\nodemddbis}{\beta}
\newcommand{\mddD}{D}
\newcommand{\elemD}{d}

\newcommand{\nPSimp}{\ell}

\newcommand{\aabs}[1]{\ensuremath{\mathring{#1}}\xspace}

\newcommand{\ncycles}{M}

\newcommand{\lcmm}{\lambda}
\newcommand{\lcmK}{\lambda^*}

\newcommand{\psol}{p'}

\newcommand{\subc}{r}

\newcommand{\cmpv}{y}

\newcommand{\sbmdd}{M}
\newcommand{\levelsb}{Z}
\newcommand{\insb}{in}

\newcommand{\esp}{w}

\newcommand{\radn}{n}
\newcommand{\radp}{p}

\newcommand{\possn}{s}
\newcommand{\possp}{o}

\begin{document}

\title{An Algorithmic Pipeline for Solving Equations over Discrete Dynamical Systems Modelling Hypothesis on Real Phenomena}

\author[unimib]{Alberto Dennunzio}
\ead{alberto.dennunzio@unimib.it}

\author[UCA]{Enrico Formenti}
\ead{enrico.formenti@unice.fr}

\author[unibo]{Luciano Margara}
\ead{luciano.margara@unibo.it}

\author[UCA]{Sara Riva}
\ead{sara.riva@univ-cotedazur.fr}

\address[unimib]{Dipartimento di Informatica, Sistemistica e Comunicazione,
  Università  degli Studi di Milano-Bicocca,
  Viale Sarca 336/14, 20126 Milano, Italy}
  
\address[UCA]{Universit\'e C\^ote d'Azur, CNRS, I3S, France}

\address[unibo]{Department of Computer Science and Engineering, University of Bologna, Cesena Campus, Via Sacchi 3, Cesena, Italy}


\begin{abstract}
This paper provides an algorithmic pipeline for studying the intrinsic structure of 
a finite discrete dynamical system (DDS) modelling an evolving phenomenon.
Here, by intrinsic structure we mean, regarding the dynamics of the DDS under observation, the feature of resulting from 
the `cooperation' of the dynamics of two or more smaller DDS. The intrinsic structure is described by an equation over DDS which represents a hypothesis over the phenomenon under observation. The pipeline allows solving such an equation, \ie, validating the hypothesis over the phenomenon, as far the asymptotic behavior and the number of states of the DDS under observation are concerned. 
The results are about the soundness and completeness
of the pipeline and they are obtained by exploiting the algebraic setting for DDS introduced in~\cite{dorigatti2018}. 
\end{abstract}
\begin{keyword}
discrete modelling, finite discrete dynamical systems, hypothesis on phenomena  
\end{keyword}
\maketitle 

\section{Introduction}
Finite Discrete Dynamical Systems (DDS for short) are useful tools that have been used since at least the seventies of the last century for modelling many evolving phenomena, especially  those rising from complex systems, that can go through a finite number of states. They  can be found in many disciplines ranging for instance from biology to chemistry, stepping through computer science,  physics, economics, sociology, \etc. Boolean automata networks, genetic regulations networks, and metabolic networks are just a few examples of DDS used in bioinformatics~\cite{bower2004computational, Sene12, Siebert2009, AracenaCS21,DemongeotMNRS22}. Cellular Automata with a finite number of cells and a finite alphabet are other examples of DDS exploited in a wide range of scientific domains for modelling complex phenomena~\cite{Adamatzky0MTTW20,AlonsoSanz12,AlvarezER08,NandiKC94,CCNC97}.    

When studying an evolving phenomenon modelled by a DDS, an important task is  describing its dynamical behavior from experimental data. 
Namely, one aims at providing a finer structure of the observed dynamics, or, in other terms, answering the following question which the central issue of this paper:
\begin{question}
\label{Q1}
Is the dynamics that we observe (from experimental
data for instance) the action of a single basic system or does it come
from the cooperation between two or more simpler systems? 
\end{question}
%
%

In this sense, a DSS can be viewed as a complex object with a certain intrinsic structure, \ie, the feature of resulting from cooperating basic components. It is crucial, of course, to precise the meaning of the word `cooperation' in Question~\ref{Q1}. In~\cite{dorigatti2018}, two forms of cooperation have been devised. The additive form, denoted by $+$, in which
two DDS with independent dynamics  provide together the observed system and the product one, denoted by $\cdot$, 
in which the observed system results from the joint parallel
action of two DDS. The formalisation of these concepts leads to endow the set of all DDS with the algebraic structure of commutative
semiring~\cite{dorigatti2018}. In this way, to face Question~\ref{Q1} it is quite natural consider multivariate monomials of the type $a\cdot x^w$ to
represent a hypothesis about a finer structure of a given DDS. Here, considering $a\cdot x^w$ means that in the first place the observed DDS is supposed to result from the joint parallel action of a known DDS, \ie, the coefficient $a$, and
$w$ copies of some yet unknown DDS $x$. Following this new point of view, a polynomial $P(x_1, \ldots, x_\qtvar)$ is hence a more complex and realistic  hypothesis on the observed DDS $b$ and then Question~\ref{Q1}
can rephrased as:

\begin{question}\label{Q:P(X)=C}
Does the dynamics that we observe result from several independent smaller systems, each of them having dynamics determined by the joint parallel action of a known part and an unknown part to be computed? In other words, does the equation $P(x_1, \ldots, x_\qtvar)=b$ have a solution? If any, what are its solutions?
\end{question}

We stress that answering this question is always possible since the constant right-hand side bounds the space of admissible solutions. 
However, it might be highly non trivial, as illustrated in~\cite{dorigatti2018} for classes of more general polynomial equations $P(x_1, \ldots , x_\qtvar) = Q(x_1, \ldots , x_\qtvar)$, although DDS have very simple dynamics. Indeed, since the set of states is finite, all their points are ultimately periodic and then any system ultimately evolves towards one of its cycle sets, each of them consisting of periodic points and called attractor, that together describe the so-called \emph{asymptotic behavior} of the system.

Since we aim at proposing a 
solid and effective tool to be used in applications that answers Question~\ref{Q:P(X)=C}, in this paper we start to focus our attention on
equations involving a multivariate polynomial without products of distinct variables.  
From now on, when no confusion is possible, we will always refer to that type of equations.
\smallskip

Our idea for solving the equation $P(x_1, \ldots, x_\qtvar)=b$ is based on three abstractions of such an equation:
the $c$-abstraction, the $a$-abstraction and
the $t$-abstraction. By means of the abstractions, potential solutions of the considered equation are filtered out. Indeed, each abstraction provides the DDSs with a specific property induced by the equation. Namely, the $c$-abstraction allows computing all DDS that satisfy the condition on the cardinality of the state set induced by the original equation.
The $a$-abstraction provides all DDS 
having a set of attractors that satisfies the equation obtained by the original one when constants and variables are restricted to their asymptotic behavior. 
Finally, the $t$-abstraction is similar to the $a$-abstraction but it  focuses on the
transient behaviour, \ie, on the states that do not belong to any attractor. We stress that the set of the solutions of the equation $P(x_1, \ldots, x_\qtvar)=b$ just turns out be the intersection of the sets of DDSs selected by these three abstractions. 
For this
reason, the enumeration of all solutions of each abstraction is needed to reach the goal.

As a first important step, in this paper we consider the $c$- and $a$-abstractions and, as results, we provide two methods solving the corresponding  abstraction equations, leaving the  $t$-abstraction for a future work.  
Let us explain their relevance.   First of all, as we will see, solving the only $a$-abstraction equation requires a non trivial pipeline (\ie, a sequence of processes with the output of one process being the input of the next one), including the computation of the $w$-th root of the asymptotic behavior of a DDS, that is a necessary intermediate step. 
Furthermore, the results allow understanding the structure of the asymptotic behavior of a phenomenon and it is well-known that this is very important in applicative scenarios. 

Both the procedures make use of \emph{Multi-valued Decision Diagrams (MDD)}(\cite{bergman2016decision, darwiche2002knowledge,bergman2014mdd}) suitably defined according to our settings. Actually, they have been applied in several domains for representing formal objects in a compressed form. Their advantage is that they perform many operations without decompressing information.
In this work, MDD are exploited to provide in an efficient way the needed solutions of all the further equations and systems derived from the abstractions, especially the $a$-abstraction. Moreover, they allow efficiently performing  some important operations that are required by our procedures as for instance the product of solutions and the intersection of sets of them. 

The paper is structured as follows. Next section introduces the background on DDS and the their semiring. The $c$- and the $a$-abstraction equations are dealt with in Section~\ref{sec:c-abstraction}
and Section~\ref{sec:a-abstraction}, respectively.
Section~\ref{sec:intersection} illustrates by a full worked out example how the solutions of the two abstraction equation can be combined.
In the last section we draw our conclusions and some perspectives.

\section{Background and basic facts}\label{back}
A \emph{Discrete Dynamical System (DDS)} is a pair $\structure{\setst,f}$ where $\setst$ is a finite \emph{set of states} and 
$f: \setst \to \setst$ is a function called \emph{next state map}. Any DDS $\structure{\setst,f}$ can be identified with the directed graph $G=\structure{V,E}$, called \emph{dynamics graph}, where $V=\setst$ and $E=\set{(\nodedds,f(\nodedds))\, |\, \nodedds\in V}$ is the graph of $f$.

\smallskip

Let $\dds$ be a DDS $\structure{\setst,f}$ and let $G$ be its dynamics graph.  If $\mathcal{Y}$ is any subset of $\setst$ such that $f(\mathcal{Y})\subseteq \mathcal{Y}$, then 
the DDS $\structure{\mathcal{Y},f|_\mathcal{Y}}$ is said to be the \emph{dynamical subsystem} of $\structure{\setst,f}$ induced by $\mathcal{Y}$ (here, $f|_\mathcal{Y}$ means the restriction of $f$ to $\mathcal{Y}$). Clearly, the dynamics graph of $\structure{\mathcal{Y},f|_\mathcal{Y}}$ is nothing but the subgraph of $G$ induced by $\mathcal{Y}$. 
A state $\nodedds\in \setst$ is a \emph{periodic point} of $\dds$ if there exists an integer $ p>0$ such that $f^{p}(\nodedds)=\nodedds$. The smallest $p$ with the previous property is called \emph{period} of $\nodedds$. If $p=1$, the state $\nodedds$ is simply a \emph{fixed point}. A \emph{cycle} (of length $p$) of $\dds$ is any set $\mathcal{C}=\{\nodedds,f(\nodedds), ..., f^{p-1}(\nodedds)\}$ where $\nodedds\in \setst$ is a periodic point of period $p$. Clearly, the set $\pp$ of all the periodic points of $\dds$ can be viewed as union of disjoint cycles. Moreover, both $\structure{\mathcal{C},f|_\mathcal{C}}$ and $\structure{\pp,f|_\pp}$ are dynamical subsystems of $\dds$ and their dynamics graphs just consist of one among, resp., all,  the strongly connected components of $G$.  
In the sequel, we will identify $\mathcal{C}$ and $\pp$ with the DDS $\structure{\mathcal{C},f|_\mathcal{C}}$ and $\structure{\pp,f|_\pp}$ (and then with their dynamics graphs too), respectively.


\smallskip
Two DDS are 
\emph{isomorphic} 
if their dynamics graph are so in the usual sense of graph theory. 
When this happens, the systems are indistinguishable from the dynamical point of view. In particular, periodic points and cycles of a system are in one-to-one correspondence with periodic points and cycles of the other system. Therefore, the dynamical subsystems induced by them in 
the respective DDS are isomorphic too.  


\smallskip
Recall that the disjoint union of two sets $\setst_1$ and $\setst_2$ is the set $\setst_1 \sqcup \setst_2 = (\setst_1 \times \{0\}) \cup (\setst_2 \times \{1\})$. 
In~\cite{dorigatti2018}, an abstract algebraic setting for DDS was introduced. In particular, the following operations over the set of DDS were defined where  the notion of disjoint union is extended to functions. 

\begin{definition}[Sum and product of DDS]\label{ddssum}
	The sum $\structure{\setst_1,f_1}+\structure{\setst_2,f_2}$ and the product $\structure{\setst_1,f_1}\cdot\structure{\setst_2,f_2}$ of any two DDS $\structure{\setst_1,f_1}$ and $\structure{\setst_2,f_2}$ are the DDS $\structure{\setst_1 \sqcup \setst_2, f_1 \sqcup f_2}$ and $\structure{\setst_1 \times \setst_2, f_1 \times f_2}$, respectively, where the function $f_1 \sqcup f_2: \setst_1 \sqcup \setst_2 \to \setst_1 \sqcup \setst_2$ is defined as: \\ 
	\begin{equation*}
		\forall (\nodedds,i) \in \setst_1 \sqcup \setst_2     \,\,\,\,  (f_1 \sqcup f_2)(\nodedds,i)=
		\begin{cases}
		(f_1(\nodedds),i)&\text{if}\;\nodedds\in \setst_1 \land i=0\\
		(f_2(\nodedds),i)&\text{if}\;\nodedds\in \setst_2  \land i=1
		\end{cases}\enspace,
	\end{equation*}
	%
while $f_1 \times f_2: \setst_1 \times \setst_2 \to \setst_1 \times \setst_2$ is the standard product of functions defined as  
$\forall (v_1,v_2) \in \setst_1 \times \setst_2, (f_1 \times f_2)(v_1,v_2)=(f_1(v_1),f_2(v_2))$ (also called \textit{direct product} in the graph literature).
\end{definition}
It is not difficult to see that the set of all DDS equipped with the sum and product operations turns out to be a semiring $R$ in which both the operations are commutative (up to an isomorphism). In the sequel, the symbols \text{\MVZero} and \text{\MVOne}  stand for their neutral elements. Moreover, for any natural $k>0$ and any DDS $\dds$, the sum $\dds+\ldots+\dds=\sum^k \dds$ and the product $\dds\cdot \ldots\cdot \dds=\prod^k \dds$ of $k$ copies of $\dds$ will be naturally denoted by $k\,\dds$ and $S^k$, respectively. In this way, we can state the following proposition which is nothing but the counterpart in our setting of the well-known standard multinomial theorem.
\begin{proposition}
\label{prop:multinomial}
For any positive naturals $\esp$, $l$, and any DDS $\dds_1, \ldots \dds_\nP$ it holds that 
\[
(\dds_1 + \ldots  + \dds_\nP)^\esp
=\sum\limits_{\substack{k_1+...+k_{l}=\esp\\0\leq k_1,\ldots,k_{l}  \leq \esp }}^{}\binom{\esp}{k_1,...,k_{l}}\prod\limits_{t=1}^{{l}}\dds_t^{k_t}
	\enspace.
\]
\end{proposition}
\smallskip

Now, consider the semiring $\ring[x_1,x_2,\ldots,x_\qtvar]$ of polynomials over $\ring$ in the variables $x_1,x_2,\ldots,x_\qtvar$, naturally induced by $\ring$. 
Polynomial equations of the following form 
model hypotheses about a certain dynamics deduced from experimental data:
\begin{equation} 
	\coef_1\cdot x_1^{\esp_1} + \coef_2\cdot x_2^{\esp_2} + \ldots + \coef_\mon \cdot x_\mon^{\esp_\mon}=\rterm
	\label{eq:problemSolvedOriginal}
\end{equation}
The known term $\rterm$ is the DDS deduced from experimental data. The coefficients $\coef_z$ (with $z\in\set{1,\ldots,\mon}$) are hypothetical DDS  
that should cooperate to produce the observed dynamics $\rterm$. Finding valid values for the unknowns in~\eqref{eq:problemSolvedOriginal}
provides a finer structure for $\rterm$ which can bring further knowledge about the observed phenomenon. 
We point out that Equation~\eqref{eq:problemSolvedOriginal} might contain duplicated pairs $(x_z, \esp_z)$ since it is the direct formulation of a hypothesis over that phenomenon. Indeed, the process of such a formulation might run into a $x_z^{\esp_z}$ which has been already considered but it has to be differently weighted.  

\section{Abstraction over the cardinality of the set of states ($c$-abstraction)}\label{sec:c-abstraction}
%
%

Given a polynomial equation over \DDS, 
a natural abstraction concerns the number of states of the DDS involved in it.
Performing such an abstraction leads to new equation in which the coefficients of the polynomial, the variables, and the constant term become those natural numbers corresponding to the cardinalities of the state sets of the DDSs involved in the original equation. 

\begin{definition}[c-abstraction]
The \emph{c-abstraction} of a DDS $\dds$ is the cardinality of its set of states. With an abuse of notation, the c-abstraction of $\dds$ is denoted by $|\dds|$.
\end{definition}

The following lemma links c-abstractions with the operations over \DDS. 

\begin{lemma}[\cite{dorigatti2018}]
For any pair of DDS $\dds_1$ and $\dds_2$, it holds that $|\dds_1+\dds_2|=|\dds_1|+|\dds_2|$ and $|\dds_1\cdot\dds_2|=|\dds_1|\cdot|\dds_2|$.
\end{lemma}

Using the notion of c-abstraction and the previous lemma,
Equation~\eqref{eq:problemSolvedOriginal} turns into the following 
\emph{$c$-abstraction equation}:

\begin{equation}\label{abstr_states}
|\coef_1|\cdot|x_1|^{\esp_1} + |\coef_2|\cdot|x_2|^{\esp_2} + . . . + |\coef_\mon|\cdot|x_\mon|^{\esp_\mon} = |\rterm|\enspace .
\end{equation}
To reach our overall goal, we need to enumerate all solutions of Equation~\eqref{abstr_states}. In this way, all possible cardinalities of the state sets of the unknown DDSs from the original  Equation~\eqref{eq:problemSolvedOriginal} will be identified. To perform that task, we proceed as follows.  
First of all, we present the enumeration problem from a combinatorial point of view.
Then, we will provide an algorithmic approach allowing the enumeration of the solutions of Equation~\eqref{abstr_states} in an efficient way.
\medskip

Let us consider the case with just one monomial (\ie, $\mon=1$) corresponding to a simpler equation of form $|\coef| \cdot |x|^\esp=|b|$ (\emph{basic case}). It is clear that
\begin{itemize}
        \item if $\esp=0$, then $|x|^\esp$ is the c-abstraction of a DDS consisting of a unique cycle of length one (a fixed point) 
        and $|\coef|=|b|$, while the equation is impossible, otherwise;
        \item if $\esp\ne0$, the equation admits a (unique) solution  
        iff $\sqrt[\esp]{|b|/|\coef|}$ is an integer number.
    \end{itemize}
Given now an equation with $\mon>1$ monomials, 
it is clear that each state of the DDS $b$ must come 
from one of them. 
Thus, we have to consider the all the ways of arranging $|\rterm|$ states among $m$ monomials. Since there can be arrangements in which not all the monomials are involved, by the Stars and Bars method (see~\cite{jongsma2019basic}, for instance), the number of such arrangements is ${|\rterm|+\mon-1 \choose \mon -1}$. Moreover, any arrangement consisting of $b_1, \ldots, b_{\mon}$ states from $\rterm$ in the respective monomials, \ie, any weak composition $b_1, \ldots, b_{\mon}$ of $|\rterm|$ into exactly $m$ parts, gives rise to the following system
	\begin{equation}
	\label{sys:aabs}
	    \begin{cases}
	        |\coef_1|\cdot|x_1|^{\esp_1}&=b_1 \\
	        |\coef_2|\cdot|x_2|^{\esp_2}&=b_2 \\
	        &\vdots\\
	        |\coef_\mon|\cdot|x_\mon|^{\esp_\mon}&=b_\mon
	    \end{cases}\enspace, 
	\end{equation}
where $\sum_{z=1}^{\mon} b_z=|\rterm|$ and each equation falls into the basic case.

Therefore, we need an efficient method that solves all feasible Systems~\eqref{sys:aabs}, \ie, those systems 
admitting a solution. Since any System~\eqref{sys:aabs} consists of equations that are all from the basic case and establishing whether each of them admits a solution is easy, the method can be designed in such a way that the space of possible solutions to be explored is reduced. 
\medskip
    
Due to the combinatorial nature of the problem,  we provide a method based on a  \emph{Multi-valued Decision Diagrams} (MDD) to enumerate the solutions of a c-abstraction equation. 
Recall that an MDD is a rooted acyclic graph able to represent a multivalued function 
having a finite set as domain and the set $\{true, false\}$ as codomain.  Both vertices and edge are labelled.  
In the structure, each level represents a variable, except for the final one with the true terminal node (called \ttnode). The first level contains the root node (called \rootnode). A path from the \rootnode to the \ttnode\ node represents a valid set of variable assignments given by the labels of the edges of that path. We stress that there can be distinct vertices (possibly on the same level) with the same label. For a sake of simplicity, we will often define the specific MDDs under the unconventional assumption that vertices form a multiset. This abuse will allow us to identify vertexes with the values of their labels. 
For more on MDD, we redirect
the interested reader to~\cite{bergman2016decision, darwiche2002knowledge,bergman2014mdd}.

\smallskip

Consider any c-abstraction equation  with $\mon$ monomials and a number $\qtvar$ of distinct variables. We associate such an equation with an MDD \structure{V,E,\lf} in which there are $\qtvar$ levels (one for each variable) and one final level for the $\ttnode$ node. 
The vertices form the multiset   $V=\sum_{i\in\{1,\ldots,\qtvar+1\}}^{} V_{i}$ where $V_{1}=\set{\rootnode}$, $V_{\qtvar+1}=\{\ttnode\}$, and  for each level  $i\in\set{2,\ldots,\qtvar}$ the set $V_{i}\subseteq\set{0,...,|\rterm|}$ of the vertexes of the level $i$ will be defined in the sequel. Indeed, the structure is built level by level. Moreover, for any node $\nodemdd\in V$, let $val(\nodemdd)=\nodemdd$
if $\nodemdd\ne\rootnode$\ and $\nodemdd\ne\ttnode$, while 
$val(\rootnode)=0$ and $val(\ttnode)=|\rterm|$. 

To define the edges outgoing from the vertexes of any level along with the corresponding labels and then the vertexes of the next level too, first of all we associate each level $i$ with the inequality $\sum_{z=1}^{\mon} var(i,z) \cdot|\coef_z|\cdot|x_z|^{\esp_z}\leq|\rterm|$, where  $var(i,z)=1$ if $|x_z|$ is the variable associated with the level $i$, $0$ otherwise, and the set 
$\mddD_{i}=\{\elemD \in \N \mid \sum_{z=1}^{\mon} var(i,z)\cdot|\coef_z|\cdot |x_z|^{\esp_z}\leq|\rterm|\text{ with }|x_z|=\elemD\}\cup\set{0}$ 
of the labels of the edges outgoing from the vertexes of the level $i$.  These labels represent the possible values for the variable corresponding to the level $i$.

Now, for each level $i \in \set{1,\ldots , \qtvar-1}$, for any vertex $\nodemdd \in V_i$ and any $\beta\in\set{0,...,|\rterm|}$ it holds that $\nodemddbis \in V_{i+1}$ and  $(\nodemdd,\nodemddbis)\in E$ iff  there exists $d\in D_i$ such that
\begin{description}
    \item $\beta=val(\nodemdd)+\sum_{z=1}^{\mon} var(i,z)\cdot|\coef_z|\cdot \elemD^{\esp_z} \leq val(\ttnode)$. 
\end{description}
Similarly, regarding the level $\qtvar$, for any vertex $\nodemdd \in V_\qtvar$, it holds that $(\nodemdd,\ttnode)\in E$ iff  there exists $d\in D_i$ such that
\begin{description}
    \item $val(\ttnode)=val(\nodemdd)+\sum_{z=1}^{\mon} var(\qtvar,z)\cdot|\coef_z|\cdot \elemD^{\esp_z}$. 
\end{description}
In both cases the edge  $(\nodemdd,\nodemddbis)$ is associated with the label  $\lfun{(\nodemdd,\nodemddbis)}=\elemD$. In this way, the labelling function  $\lf\colon E\to \bigcup_{i\in\{1,\ldots,\qtvar\}}^{} \mddD_{i}$ has been defined too. 

The value $val(\nodemdd)$ associated with any node $\nodemdd$ represents the amount of states obtained from a partial set of variable assignments, \ie, a set of assignments involving the variables until the level the node $\nodemdd$ belongs to,  each of them  corresponding to a path from $\rootnode$ to $\nodemdd$.  The c-abstraction equation admits no solution if there is no path  from $\rootnode$ to $\ttnode$ on the associated MDD. 

Finally, the MDD is reduced by performing a pReduction \ie a procedure that merges equivalent nodes (on the same layer)
and delete all nodes (and the corresponding edges) which are not on a path from \rootnode\ to \ttnode~\cite{perez2015efficient}. 
\begin{example}\label{ex:states}
    Consider the following equation:
    \[
    2\cdot|x_3|+5\cdot|x_1|^2 + 4\cdot|x_2|+4\cdot|x_1|^4+4\cdot|x_3|^2 = 593 \enspace.
    \]
    Hence, there are ${593+5-1 \choose 5 -1}=5.239776465 \times 10^9$ way of arranging $|\rterm|=593$ states among $m=5$ monomials. 
    However, not all of them 
    give rise to 
    solutions 
    of that equation. 
    According to the definition, the resulting reduced MDD is 
    illustrated in Figure~\ref{fig:MDDnodes}. The first level of the structure represents the possible values 
    for the variable $|x_1|$. In the second one, the red edges along with the corresponding label represent the possible values 
    for $|x_2|$, in the case 
    $|x_1|=1$, while the blue ones are the possible assignments for $|x_2|$, in the case 
    $|x_1|=3$. The last edge layer 
    represents the possible values for $|x_3|$.
    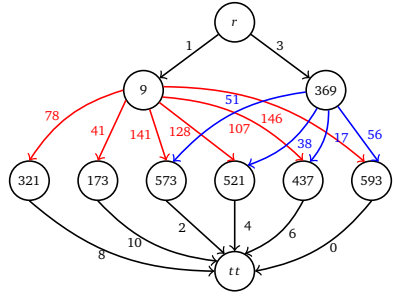
\begin{figure}[t]
  \begin{center}
	\begin{tikzpicture}[scale=.3,-latex,auto, semithick , state/.style ={circle,draw,minimum width=8mm},font=\footnotesize]
	\node[state,scale=.65] (r) at (0,0) {$r$};
	\node[state,scale=.65] (1=9) at (-4,-3) {$9$};
	\node[state,scale=.65] (1=369) at (4,-3) {$369$};
	\node[state,scale=.65] (2=321) at (-9,-7) {$321$};
	\node[state,scale=.65] (2=173) at (-6,-7) {$173$};
	\node[state,scale=.65] (2=573) at (-3,-7) {$573$};
	\node[state,scale=.65] (2=521) at (0,-7) {$521$};
	\node[state,scale=.65] (2=437) at (3,-7) {$437$};
	\node[state,scale=.65] (2=593) at (6,-7) {$593$};
    \node[state,scale=.65] (tt) at (0,-11) {$tt$};
    
    \draw [->]  (r) to node[above,scale=.65] {$1$} (1=9);
    \draw [->]  (r) to node[above,scale=.65] {$3$} (1=369);
    \path[->,red] (1=9.180)  edge [bend right=20,left] node {\scalebox{0.65}{$78$}} (2=321.90);
    \draw [->,red]  (1=9.210) to node[left,scale=.65] {$41$} (2=173.90);
    \draw [->,red]  (1=9) to node[left,scale=.65] {$141$} (2=573.90);
    \draw [->,red]  (1=9) to node[left,scale=.65] {$128$} (2=521.90);
    \path[->,red] (1=9.340)  edge [bend left=20,below] node {\scalebox{0.65}{$107$}} (2=437.90);
    \path[->,red] (1=9.10)  edge [bend left=20,below] node {\scalebox{0.65}{$146$}} (2=593.110);
    \path[->,blue] (1=369)  edge [right] node {\scalebox{0.65}{$56$}} (2=593.70);
    \path[->,blue] (1=369)  edge [bend left=20,right] node {\scalebox{0.65}{$17$}} (2=437.70);
    \path[->,blue] (1=369)  edge [bend left=20,right] node {\scalebox{0.65}{$38$}} (2=521.50);
    \path[->,blue] (1=369)  edge [bend right=20,above] node {\scalebox{0.65}{$51$}} (2=573.70);
    
    \path[->] (2=321.270)  edge [bend right=20,left] node {\scalebox{0.65}{$8$}} (tt.180);
    \path[->] (2=173.270)  edge [bend right=20,left] node {\scalebox{0.65}{$10$}} (tt.150);
    \path[->] (2=573.270)  edge [left] node {\scalebox{0.65}{$2$}} (tt.120);
    \path[->] (2=521.270)  edge [right] node {\scalebox{0.65}{$4$}} (tt.90);
    \path[->] (2=437.270)  edge [bend left=20,right] node {\scalebox{0.65}{$6$}} (tt.60);
    \path[->] (2=593.270)  edge [bend left=20,right] node {\scalebox{0.65}{$0$}} (tt.0);
    \end{tikzpicture}
  \end{center}
	\caption{The reduced MDD representing all the solutions of $2\cdot|x_3|+5\cdot|x_1|^2 + 4\cdot|x_2|+4\cdot|x_1|^4+4\cdot|x_3|^2 = 593$. There are $\qtvar=3$ variables, which are represented in the structure in the following order: $|x_1|$, $|x_2|$, and $|x_3|$.}
	\label{fig:MDDnodes}
\end{figure}
    \end{example}
    %
    We stress that the MDD allows the exploration of the solution space of the equation in a efficient way. In fact, at each level only a part of the possible values for a variable are considered depending 
    on the feasible assignments of the variables of the previous levels. 
    Moreover, the MDD can 
    gain up to an exponential factor in representation space through the reduction process.
    The worst case space complexity is $O(|\rterm|\qtvar+\delta)$, in terms of number of nodes and edges, where $\delta=\sum_{i=1}^{\qtvar} |D_i|$. The p-reduction reduces the total number of edges to $\delta'\ll\delta$ and the bound of the number of nodes of any level to $\mu\leq|\rterm|$, giving rise to a lower complexity $O(\mu\qtvar+\delta')$. Actually, this bound is never reached in our experiments. As an illustrative case, consider Example~\ref{ex:states}. The MDD could have up to $1188$ nodes and $352835$ edges, but its reduced version has only  $10$ nodes and $18$ edges (see Figure~\ref{fig:MDDnodes}).
    \medskip
    
    Let us recall that the equation over c-abstractions is a polynomial equation over natural numbers. Therefore, simplifications are possible and the whole approach can be applied to the simplified equation.

\section{Abstraction over the asymptotic behaviour ($a$-abstraction)}\label{sec:a-abstraction}
In this section we deal with a further abstraction, namely, the asymptotic one, describing the long-term behaviour of a DDS, \ie, its ultimate periodic behaviour. In particular, we provide a method for solving the version of Equation~\eqref{eq:problemSolvedOriginal} obtained considering the asymptotic behaviour of constants and variables.

\begin{notation}
In the sequel, for any pair of positive integers $n$ and $p$, $C^{n}_{p}$ will stand for the union of any $n$ disjoint cycles of length $p$ of a DDS $\dds$. To stress that we are dealing with sets consisting of union of disjoint cycles, each of them identifying a dynamical subsystem of $\dds$, the operations of disjoint union and   product of two of such sets, or, by identification, the sum and product of the corresponding dynamical (sub)systems, will be denoted by $\oplus$ and $\odot$ instead of $+$ and $\cdot$, respectively. 
According to this notation, it is clear that $C^{n_1}_{p} \oplus C^{n_2}_{p} =C^{n_1+n_2}_{p}$ for any pair of positive naturals $n_1, n_2$ and  $kC^{n}_{p}=C^{kn}_{p}$ for any positive natural $k$. Finally, for any positive natural $i$ and any positive naturals $p_{1}, \ldots, p_{i}$, denote $\lcmm_i=\lcm(p_1, \ldots, p_{i})$.
\end{notation}
\begin{definition}[a-abstraction]
\label{def:abs}
    The \emph{a-abstraction} of a \DDS $\dds$, denoted by $\aabs{\dds}$, is the dynamical subsystem of $\dds$ induced by the set $\pp$ of all its periodic points, or, by identification, the set  $\pp$ itself. 
%
\end{definition}
\begin{remark}
\label{rem:abs}
It immediately follows from the previous definition that the a-abstraction of the sum, resp., the product, of two DDS, is the sum, resp., the product of the a-abstractions of the two DDS. Moreover, the a-abstraction of a DDS $\dds$ can be written as $$\aabs{\dds}=\bigoplus\limits_{i=1}^\nP C^{n_{i}}_{p_{i}}\enspace,$$ for some positive naturals $\nP$, $n_{1},\ldots, n_{\nP}$, and pairwise distinct positive naturals $p_{1}, \ldots, p_{\nP}$, 
where, 
for each $i \in \{1, ...,\nP\}$, $n_i$ is the number of disjoint cycles of length $p_i$ (see Figure~\ref{exass} for an illustrative example).
\end{remark}

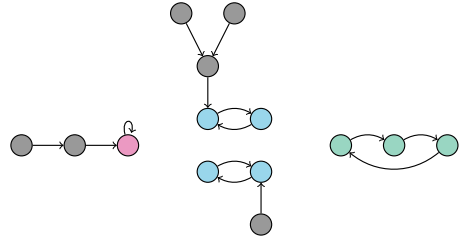
\begin{figure}[]
	\centering
	\scalebox{.7}{
	\begin{tikzpicture}[scale=.5,-latex,auto, semithick , font=\footnotesize, state/.style ={circle,draw,minimum width=4mm}]
 	\node[state,fill=rose!40] (a) at (0,0) {};
 	\node[state,fill=black!40] (a_t) at (-2,0) {};
 	\node[state,fill=black!40] (a_t2) at (-4,0) {};
	 \node[state,fill=vert!40] (d) at (8,0) {};
	 \node[state,fill=vert!40] (e) at (10,0) {};
	 \node[state,fill=black!40] (b_t) at (3,3) {};
	 \node[state,fill=black!40] (b_t2) at (2,5) {};
	 \node[state,fill=black!40] (b_t3) at (4,5) {};
	 \node[state,fill=bleu!40] (b) at (3,1) {};
	 \node[state,fill=bleu!40] (c) at (5,1) {};
	 \node[state,fill=bleu!40] (g) at (3,-1) {};
	 \node[state,fill=bleu!40] (h) at (5,-1) {};
	 \node[state,fill=black!40] (h_t) at (5,-3) {};
	 \node[state,fill=vert!40] (f) at (12,0) {};
    \path (a) edge [loop above] node {} (a);
	\draw[->] (b) to[bend left] (c);
	\draw[->] (a_t) to[] (a);
	\draw[->] (a_t2) to[] (a_t);
	\draw[->] (b_t) to[] (b);
	\draw[->] (b_t2) to[] (b_t);
	\draw[->] (b_t3) to[] (b_t);
	 \draw[->] (c) to[bend left] (b);
	\draw[->] (g) to[bend left] (h);
	 \draw[->] (h) to[bend left] (g);
	 \draw[->] (h_t) to[] (h);
	 \draw[->] (d) to[bend left] (e);
	 \draw[->] (e) to[bend left] (f);
	 \draw[->] (f) to[bend left=40] (d);
	\end{tikzpicture} 	}	
	\caption{A DDS with four cycles ($\nP=3$): $(\textcolor{rose}{C^1_1} \oplus \textcolor{bleu}{C^2_2} \oplus \textcolor{vert}{C^1_3})$ in our notation. }
	\label{exass}
\end{figure}

The following proposition provides an explicit expression for the product of several unions of cycles. It will be very useful in the sequel.  
\begin{proposition}\label{prop:prod}
For any natural $\nP>1$ and any positive naturals $n_{1},\ldots, n_{\nP}$, $p_{1}, \ldots, p_{\nP}$, it holds that
\begin{equation*}
\begin{aligned}
        \bigodot\limits_{i=1}^{\nP}C^{n_i}_{p_i} = C^{\frac{1}{\lcmm_{\nP}}{\prod_{i=1}^{\nP} (p_i n_i)}}_{\lcmm_{\nP}}\enspace.
\end{aligned}
 \end{equation*}
\end{proposition}
\begin{proof}
We proceed by finite induction over $\nP$.  
First of all, we prove that the statement is true for $\nP=2$, i.e., 
\begin{equation}
\label{due}
C^{n_1}_{p_1}\odot C^{n_2}_{p_2} = C^{\frac{1}{\lcmm_2}\cdot{p_1 n_1  \cdot p_2 n_2}}_{\lcmm_2}\enspace.
\end{equation}
Let us consider the case $n_1=n_2=1$. Since $C^{1}_{p_1}$ and $C^{1}_{p_2}$ can be viewed as finite cyclic groups of order $p_1$ and $p_2$, respectively, each element of the product of such cyclic groups has order $\lcm(p_1, p_2)$ or, in other words, each element of $C^{1}_{p_1}\odot C^{1}_{p_2}$ belongs to some cycle of length $\lcmm_2$. So, $C^{1}_{p_1}\odot C^{1}_{p_2}$ just consists of  $(p_1\cdot p_2)/\lcmm_2$ cycles, all of length $\lcmm_2$, and therefore $$C^{1}_{p_1}\odot C^{1}_{p_2}=C^{\frac{1}{\lcmm_2}\cdot{p_1   \cdot p_2 }}_{\lcmm_2}\enspace.$$ In the case $n_1\neq 1$ or $n_2\neq 1$, since the product is distributive over the sum, we get
\[
C^{n_1}_{p_1}\odot C^{n_2}_{p_2} = \bigoplus\limits_{i=1}^{n_1} C^{1}_{p_1} \odot\; \bigoplus\limits_{j=1}^{n_2}  C^{1}_{p_2} = \bigoplus\limits_{i=1}^{n_1} \bigoplus\limits_{j=1}^{n_2}  (C^{1}_{p_1} \odot C^{1}_{p_2})=
\bigoplus\limits_{i=1}^{n_1} \bigoplus\limits_{j=1}^{n_2}  C^{\frac{1}{\lcmm_2}\cdot{p_1   \cdot p_2 }}_{\lcmm_2}=
C^{\frac{1}{\lcmm_2}\cdot{p_1 n_1  \cdot p_2 n_2}}_{\lcmm_2}\enspace.
\]
Assume now that the equality holds for any $\nP>1$. Then, we get
\[
\bigodot\limits_{i=1}^{\nP+1}C^{n_i}_{p_i} = C^{\frac{1}{\lcmm_{\nP}}{\prod_{i=1}^{\nP} (p_i n_i)}}_{\lcmm_{\nP}} \odot C^{n_{\nP+1}}_{p_{\nP+1}}= C^{\frac{1}{\lcm(\lcmm_{\nP}, p_{\nP+1} )}\cdot {\prod_{i=1}^{\nP} (p_i n_i)}\cdot (p_{\nP+1}n_{\nP+1}) }_{\lcm(\lcmm_{\nP}, p_{\nP+1} )}
= C^{\frac{1}{\lcmm_{\nP+1}}\cdot {\prod_{i=1}^{\nP+1} (p_i n_i)}}_{\lcmm_{\nP+1}}
\enspace.
\]
Therefore, the equality  also holds for $\nP+1$ and this concludes the proof.
\end{proof}
We now consider the $w$-th power of the union of cycles of a certain lengths and the $w$-th power of the sum of such unions.  Before proceeding, for any DDS $\dds$, we naturally define $\dds^0$ as $C^1_1$, \ie, the neutral element \text{\MVOne} of the product operation. %
\begin{corollary}
\label{lem:s1}
For any natural numbers $\esp\geq 1$, $n\geq 1$, and $p\geq1$, it holds that:
\begin{equation*}
    (C^{n}_{p})^\esp=C^{p^{\esp-1}n^\esp}_{p} .
\end{equation*}
\end{corollary}
\begin{proof}
It is an immediate consequence of Proposition~\ref{prop:prod}. 
\end{proof}
\begin{proposition}
\label{general_newt}
For any positive naturals  $\nP>1$, $w>1$, $n_{1},\ldots, n_{\nP}$, and $p_{1}, \ldots, p_{\nP}$, it holds that
\[
\left(\bigoplus \limits_{i=1}^{\nP} C^{n_i}_{p_i} \right)^\esp
=\bigoplus\limits_{\substack{k_1+...+k_{\nP}=\esp\\0\leq k_1,\ldots,k_{\nP}  \leq \esp }}^{}\binom{\esp}{k_1,\ldots,k_{\nP}}C^{\frac{1}{\lcmK_\nP}\cdot \prod_{i=1}^{\nP} (p_i n_i)^{k_i}}_{\lcmK_\nP}
\]
where, for any tuple $k_1,\ldots, k_i$,  
$\lcmK_i$ is the $\lcm$ of those $p_j$ with $j\in\set{1, \ldots, i}$ and $k_j \not = 0$ (while $\lcmK_i=1$ iff all $k_j=0$). 
\end{proposition}
\begin{proof}
By Proposition~\ref{prop:multinomial}, Proposition~\ref{prop:prod}, and Corollary~\ref{lem:s1}, we get 
\begin{align*}
\left(\bigoplus \limits_{i=1}^{\nP} C^{n_i}_{p_i} \right)^\esp
&=\bigoplus\limits_{\substack{k_1+\ldots+k_{\nP}=\esp\\0\leq k_1,\ldots,k_{\nP}  \leq \esp }}^{}\binom{\esp}{k_1,\ldots,k_{\nP}}\bigodot\limits_{t=1}^{{\nP}}(C^{\radn_t}_{\radp_t})^{k_t}\\
&=\bigoplus\limits_{\substack{k_1+\ldots+k_{\nP}=\esp\\0\leq k_1,\ldots,k_{\nP}  \leq \esp }}^{}\binom{\esp}{k_1,\ldots ,k_{\nP}}\bigodot\limits_{t=1, k_t\neq 0}^{{\nP}}C^{{\radp_t}^{k_t-1}{\radn_t}^{k_t}}_{\radp_t}\\
&=
\bigoplus\limits_{\substack{k_1+\ldots+k_{\nP}=\esp\\0\leq k_1,\ldots,k_{\nP}  \leq \esp }}^{}\binom{\esp}{k_1,\ldots,k_{\nP}}C^{\frac{1}{\lambda^*_{\nP}}\cdot \left(\prod_{\substack{t=1\\k_t \neq0}}^{\nP} (\radp_t^{k_t} \radn_t^{k_t})\right) }_{\lambda^*_{\nP}}\\
&= \bigoplus\limits_{\substack{k_1+...+k_{\nP}=\esp\\0\leq k_1,\ldots,k_{\nP}  \leq \esp }}^{}\binom{\esp}{k_1,\ldots,k_{\nP}}C^{\frac{1}{\lcmK_\nP}\cdot \prod_{i=1}^{\nP} (p_i n_i)^{k_i}}_{\lcmK_\nP}\enspace.
\end{align*}
\end{proof}
We can now write the \emph{a-abstraction equation} obtained by considering just the asymptotic behavior of all constants and variables in~Equation~\eqref{eq:problemSolvedOriginal}:
\begin{equation}
\label{eq:aAbsEq}
\mathring{\coef}_1\cdot \mathring{x}_1^{\esp_1} + \ldots + \mathring{\coef}_\mon \cdot \mathring{x}_\mon^{\esp_\mon}=\mathring{\rterm}\enspace,
\end{equation}
where, according to Remark~\ref{rem:abs}, for each $z\in\{1, \ldots , \mon\}$ the a-abstraction of the coefficient $\coef_z$ and the a-abstraction of the known term $\rterm$ are
\[
\mathring{\coef}_z= \bigoplus\limits_{i=1}^{\nP_z} C_{p_{zi}}^{n_{zi}} \qquad \text{and} \qquad \mathring{\rterm}= \bigoplus\limits_{j=1}^{\Pright} C_{p_{j}}^{n_{j}}\enspace.
\]
To solve the a-abstraction equation, we first carry out some simplifications. First of all, we consider the actual number $\mons\leq \mon$ of distinct pairs $(\mathring{x}_z, w_z)$ appearing in such an equation. In this way, Equation~\eqref{eq:aAbsEq} can be rewritten as 
\begin{equation}
    \bigoplus\limits_{i=1}^{\nPSimp_1} C_{p_{1i}}^{n_{1i}} \odot X_1 \oplus\ldots \oplus \bigoplus\limits_{i=1}^{\nPSimp_\mons} C_{p_{\mons i}}^{n_{\mons  i}} \odot X_\mons = \bigoplus\limits_{j=1}^{\Pright} C_{p_{j}}^{n_{j}}\enspace,
    \label{eq:simplified}
\end{equation}
where, for each $z \in \{1, \ldots,\mons\}$,  $X_z$ denotes $\mathring{x}_z^{\esp_z}$,    $\nPSimp_z$ is the number of the distinct lengths of the cycles forming the coefficient of $X_z$,  
and, with an abuse of notation, the number $n^{\prime}_{zi}$ of cycles of length $p_{zi}$ inside that coefficient is still denoted by $n_{zi}$ even though it may hold that $n^{\prime}_{zi}\neq n_{zi}$.
\medskip

Equation~\eqref{eq:simplified} is still hard to solve in this form. 
We can further simplify it  by performing
a \textbf{contraction step}
which consists in rewriting it in an equivalent way as union of systems of the following type, one for each 
vector $(n^{11}_{1}, \ldots, n^{11}_{\Pright})$ obtained varying each $n^{11}_{j}\in\set{0,\ldots,n_j}$ with    $j\in\set{1,\ldots,\Pright}$:
\begin{subequations}
  \begin{empheq}[left=\empheqlbrace]{align}
C_{p_{11}}^{n_{11}} \odot 
X_1 
= \bigoplus\limits_{j=1}^{\Pright} C_{p_{j}}^{n^{11}_{j}}\label{eq:contraction-a}\\
C_1^1 \odot \mathring{y} = \bigoplus\limits_{j=1}^{\Pright} C_{p_{j}}^{n_{j}-n^{11}_{j}}\label{eq:contraction-b}
 \end{empheq}
\end{subequations}
where $\mathring{y}=(\bigoplus\limits_{i=2}^{\nPSimp_1} C_{p_{1i}}^{n_{1i}} \odot 
X_1
) \oplus(\bigoplus\limits_{i=1}^{\nPSimp_2} C_{p_{2i}}^{n_{2i}} \odot 
X_2
) \oplus\ldots \oplus (\bigoplus\limits_{i=1}^{\nPSimp_\mons} C_{p_{\mons i}}^{n_{\mons  i}} \odot 
X_{\mons}
)
$.

At this point, let us repeat as long as possible the application of the contraction step over the last equation of each system obtained by the previous contraction step. We stress that such an application essentially consists in 
\begin{enumerate}
\item[$i)$] updating $\mathring{y}$ by removing a term $C_{p_{z i}}^{n_{zi}} \odot 
X_z
$ with $z\in\set{1,\ldots,\mons}$ and $i\in\set{1,\ldots,\nPSimp_z}$, 
\item[$ii)$] considering all possible vectors $(n^{z i}_{1}, \ldots, n^{z i}_{\Pright})$ obtained varying each $n^{zi}_{j}$  
with $j\in\set{1,\ldots,\Pright}$ from $0$ to the remaining number of cycles of length $p_j$ of the right-hand side,
\item[$iii)$]
 introducing, for each of the above mentioned vectors, a new system obtained by adding the following equation  
 \[
 C_{p_{z i}}^{n_{zi}} \odot
 X_z 
 = \bigoplus\limits_{j=1}^{\Pright} C_{p_{j}}^{n^{z i}_{j}}
 \]
to the considered initial system just before the equation involving $\mathring{y}$.
\item[$iv)$] updating the right-hand side of the equation involving $\mathring{y}$ by removing $n^{zi}_j$ cycles from the unions of cycles of length $p_j$.
\end{enumerate}
In this way, we eventually get that Equation \eqref{eq:simplified} can be equivalently rewritten as a union of systems, each of them having the following form 
\begin{equation}
    \label{systelcontr}
    \begin{cases}
C_{p_{11}}^{n_{11}} \odot X_1 & = \bigoplus\limits_{j=1}^{\Pright} C_{p_{j}}^{n^{11}_{j}}\\
C_{p_{12}}^{n_{12}} \odot X_1 & = \bigoplus\limits_{j=1}^{\Pright} C_{p_{j}}^{n^{12}_{j}}\\
&\vdots\\
C_{p_{1\nPSimp_1}}^{n_{1\nPSimp_1}} \odot X_1 & = \bigoplus\limits_{j=1}^{\Pright} C_{p_{j}}^{n^{1\nPSimp_1}_{j}}\\
C_{p_{21}}^{n_{21}} \odot X_2 & = \bigoplus\limits_{j=1}^{\Pright} C_{p_{j}}^{n^{21}_{j}}\\
&\vdots\\
C_{p_{\mons \nPSimp_\mons}}^{n_{\mons \nPSimp_\mons}} \odot X_\mons & = \bigoplus\limits_{j=1}^{\Pright} C_{p_{j}}^{n^{\mons \nPSimp_\mons}_{j}}
\end{cases}\enspace.
\end{equation}
Referring to Equation~\eqref{eq:simplified}, we stress that, for each $j \in \set{1,...,\Pright}$, it holds that   the number of cycles of length $p_j$ involved in know term  is just  $n_j=\sum_{z=1}^{\mons} \sum_{i=1}^{\nPSimp_z} n^{zi}_j$,   where $n^{zi}_j$ represents the number of those that 
the monomial $C_{p_{z i}}^{n_{z i}} \odot X_z$ contributes to form. 

Now, to solve any equation 
$C^{n_{zi}}_{p_{zi}} \odot X_z = \bigoplus\limits_{j=1}^{\Pright} C_{p_{j}}^{n^{zi}_{j}}$ from \eqref{systelcontr}, it is enough to solve the following  $\Pright$ equations  
\begin{equation}
\label{eq:tanteq}
C^{n_{zi}}_{p_{zi}} \odot X_z=C_{p_{1}}^{n^{zi}_{1}},\quad \ldots\quad,\quad C^{n_{zi}}_{p_{zi}} \odot X_z=C_{p_{\Pright}}^{n^{zi}_{\Pright}}
\end{equation}
and compute the Cartesian product among their solutions. Since, for each $j\in\{1,\ldots, \Pright\}$, equation $C^{n_{zi}}_{p_{zi}} \odot X_z=C_{p_{j}}^{n^{zi}_{j}}$ can be rewritten as $$C^{1}_{p_{zi}} \odot X_z = C^{n^{zi}_{j}/{n_{zi}}}_{p_j}\enspace,$$
if $n^{zi}_{j}/{n_{zi}}$ is a natural number, while it has no solution, otherwise, 
solving Equation~\eqref{eq:simplified} reduces to identify all the Systems \eqref{systelcontr} and perform the products and intersections of the solutions of a certain number of simpler equations, called \textbf{basic equations}, with the following form:
\begin{equation} 
\label{eq:simple}
	C^1_p \odot X = C^n_q\enspace,
\end{equation}
where $X$ is some $X_z$,  $p\in\set{p_{11},p_{12}, \ldots,p_{\mons \nPSimp_\mons}}$, $q \in \set{p_1,...,p_{\Pright}}$, and, making reference to the right-hand side, $n$ is smaller or equal to $n_j$, \ie, the number of cycles of length $q=p_j$.
\medskip

To solve Equation~\eqref{eq:simplified},  we need an efficient method that: 1) enumerates the solutions of all Equations \eqref{eq:simple}, \ie, the values of $X_z$, 2) computes the suitable products of these solutions and the intersections of sets of them, 3) retrieves the value of $\mathring{x_z}$ from $X_z$.
The algorithmic pipeline illustrated in Figure~\ref{fig:pipeline} just performs all these tasks. Since a finite but potentially large number of basic equations have to be solved, the pipeline is designed in order that first of all the basic equations admitting solution are identified. In this way, Systems~\eqref{systelcontr} involving basic equations without solutions are  avoided, or, in other words, only feasible contraction steps, \ie, feasible Systems~\eqref{systelcontr} generated by contraction steps are considered. 
%
An MDD-based technique that 
enumerates the solutions of any basic equation is illustrated in Section \ref{mddsec} (task 1), while the identification of all the feasible contraction steps is  presented in 
Section~\ref{cssec} along with the way of solving their corresponding feasible systems starting from the  solutions of basic equations (task 2). Finally, Section~\ref{roots} explains how to compute the DDS $\mathring{x_z}$ starting  from the solutions $X_z$ (task 3). 

\medskip

\usetikzlibrary{positioning,
                shadows,
                shapes.multipart}
\pgfdeclarelayer{foreground}
\pgfdeclarelayer{background}
\pgfsetlayers{background,main,foreground}
\makeatletter
\def\tikz@extra@preaction#1{
  {%
    \pgfsys@beginscope%
      \setbox\tikz@figbox=\box\voidb@x%
      \begingroup\tikzset{#1}\expandafter\endgroup%
      \expandafter\def\expandafter\tikz@preaction@layer
\expandafter{\tikz@preaction@layer}%
      \ifx\tikz@preaction@layer\pgfutil@empty%
      \path[#1];
      \else%
      \begin{pgfonlayer}{\tikz@preaction@layer}%
      \path[#1];%
      \end{pgfonlayer}
      \fi%
      \pgfsyssoftpath@setcurrentpath\tikz@actions@path
      \tikz@restorepathsize%
    \pgfsys@endscope%
  }%
}
\let\tikz@preaction@layer=\pgfutil@empty
\tikzset{preaction layer/.store in=\tikz@preaction@layer}
\makeatother
\newcommand{\midarrow}{\tikz \draw[-triangle 90] (0,0) -- +(.1,0);}
\tikzset{
mpv/.style = {
    rectangle split,
    rectangle split parts=2,
    rectangle split part fill={#1}, 
    draw, rounded corners,
    align=center, text=black,
  preaction layer=background,       
    drop shadow}, 
        }
\tikzstyle{small} = [text width=3.5cm]
\tikzstyle{big} = [text width=5cm]
\tikzstyle{bigbig} = [text width=7cm]
\tikzset{->-/.style={thick,>=stealth,decoration={
  markings,
  mark=at position .5 with {\arrow{>}}},postaction={decorate}}}
\definecolor{alizarin}{rgb}{0.82, 0.1, 0.26}
\definecolor{ballblue}{rgb}{0.13, 0.67, 0.8}
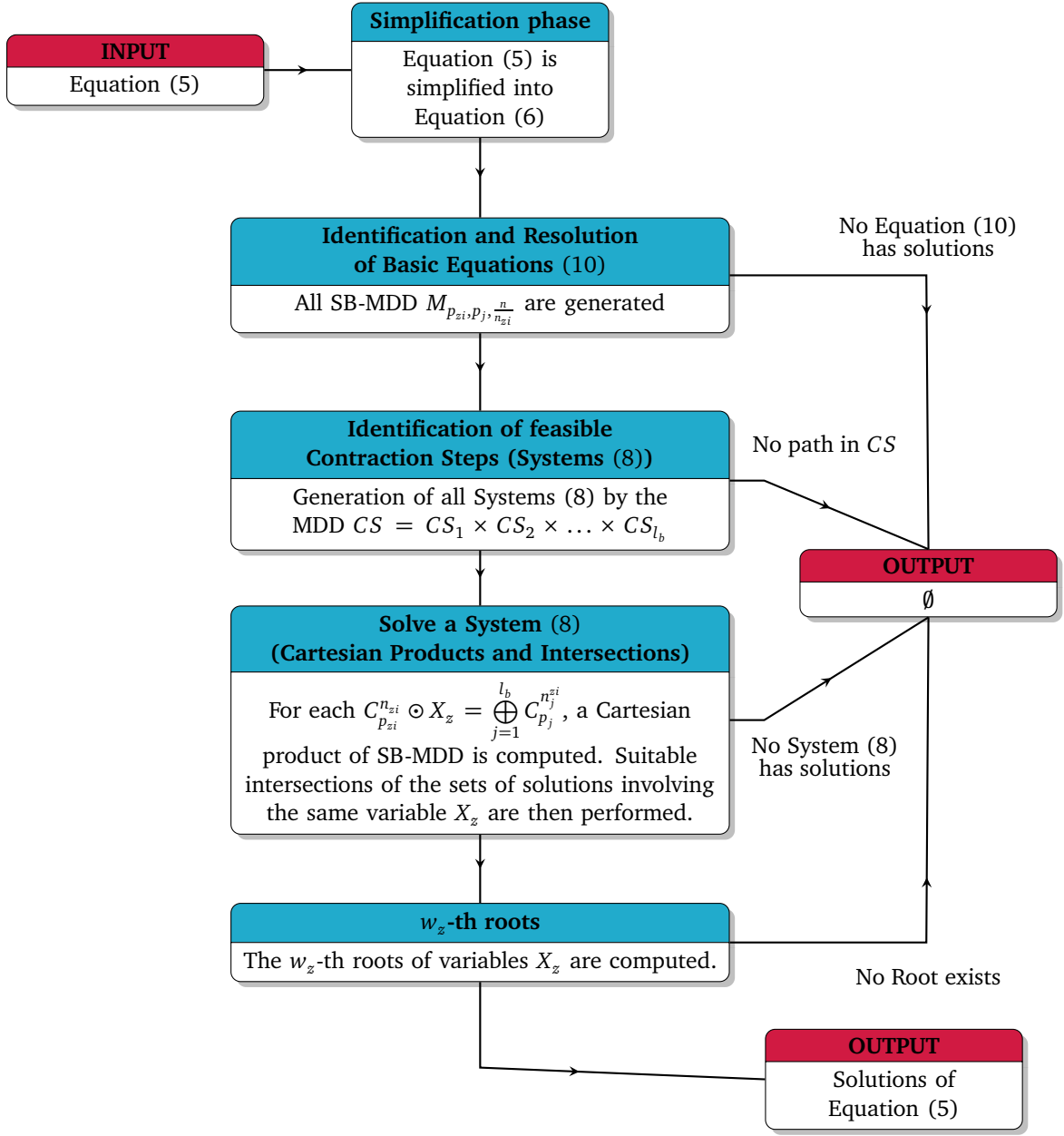
\begin{figure}[]
    \centering
    \begin{tikzpicture}[node distance=2cm]
     \node (input) [mpv={alizarin,white},small] {
        \textbf{INPUT}
        \nodepart{two}  Equation \eqref{eq:aAbsEq}
    }; 
     \node (semp) [right of=input, mpv={ballblue,white}, xshift=3cm,small]{
        \textbf{Simplification phase}
        \nodepart{two}  Equation \eqref{eq:aAbsEq} is \mbox{simplified} into Equation \eqref{eq:simplified}
    }; 
    \draw [->-] (input) --  (semp);
     \node (basic) [mpv={ballblue,white},bigbig, below of=semp,yshift=-1cm]  {
        \textbf{Identification and \mbox{Resolution} of Basic Equations \eqref{eq:simple}}
        \nodepart{two} All SB-MDD $M_{p_{zi},p_{j},\frac{n}{n_{zi}}}$ are generated
    };
    \node[right of=basic,xshift=4.5cm,yshift=0.7cm] {No Equation \eqref{eq:simple}};
    \node[right of=basic,xshift=4.5cm,yshift=0.4cm] {has solutions};
    \node (nobasic)
    [mpv={alizarin,white},small, right of=basic,xshift=4.5cm,yshift=-4.5cm] {
        \textbf{OUTPUT}
        \nodepart{two}  
        $\emptyset$
    };
    \draw [->-] (semp) -- (basic);
    \draw [->-] (basic.east) -| ++(2.85,0) -- (nobasic.north);
    \node (cs) [mpv={ballblue,white},bigbig, below of=basic,yshift=-1cm] {
        \textbf{Identification of feasible \mbox{Contraction} Steps (Systems \eqref{systelcontr})}
        \nodepart{two}  Generation of all Systems \eqref{systelcontr} by the MDD $CS=CS_1 \times CS_2 \times \ldots \times CS_{\Pright}$
    };
    \node[right of=cs,xshift=3cm,yshift=0.5cm] {No path in $CS$};
    \draw [->-] (basic) -- (cs);
    \draw [->-] (cs.east) -| ++(.5,0) -- (nobasic.north);
    \node (solvecs) [mpv={ballblue,white},bigbig, below of=cs,yshift=-1.5cm] {
        \textbf{Solve a System \eqref{systelcontr} (\mbox{Cartesian Products} and Intersections)}
        \nodepart{two}  For each $C^{n_{zi}}_{p_{zi}} \odot X_z = \bigoplus\limits_{j=1}^{\Pright} C_{p_{j}}^{n^{zi}_{j}}$, a Cartesian product of SB-MDD is computed. Suitable intersections of the sets of solutions involving the same variable $X_z$ are then performed.
    };
    \node[right of=solvecs,xshift=3cm,yshift=-0.4cm] {No System \eqref{systelcontr}};
    \node[right of=solvecs,xshift=3cm,yshift=-0.7cm] {has solutions};
    \draw [->-] (cs) -- (solvecs);
     \node (roots) [mpv={ballblue,white},bigbig, below of=solvecs,yshift=-1.25cm] {
        \textbf{$\esp_z$-th roots 
        }
        \nodepart{two} The $\esp_z$-th roots of variables $X_z$ are computed. 
    };
    \node[right of=roots,xshift=4.5cm,yshift=-0.5cm] {No Root exists};
    \draw [->-] (solvecs) -- (roots);
    \draw [->-] (solvecs.east) -| ++(.5,0) -- (nobasic.south);
    \node (output) [mpv={alizarin,white},small, right of=roots,xshift=4cm,yshift=-2cm] {
        \textbf{OUTPUT}
        \nodepart{two}  
        Solutions of Equation \eqref{eq:aAbsEq}
    };
    \draw [->-] (roots.east) -| ++(2.85,0) -- (nobasic.south);
    \draw [->-] (roots.south) -| ++(0,-1.25) -- (output.west);
    \end{tikzpicture}
    \caption{The MDD-based algoritmic pipeline for solving an $a$-abstraction equation.}
    \label{fig:pipeline}
\end{figure}



\subsection{An MDD-based method for solving a basic equation } \label{mddsec}
In this section we are going to solve a basic equation by means of a suitable MDD\footnote{This subsection and the next one are an improved version of the conference paper~\cite{riva2021mdd}.}.
Let us start by considering any basic equation $C^1_p \odot X=C^n_q$. According to Remark~\ref{rem:abs}, each of its solutions is expressed as a sum of unions of disjoint cycles.
\smallskip

By Proposition \ref{prop:prod}, among a certain number of cycles all of length $\psol$ and that form
an addend of  a 
solution, one cycle of length $\psol$ gives rise to $\subc$ cycles of length $q$  
inside $C^n_q$ when it is multiplied by $C^1_p$ iff $\subc$ divides $q$, $\psol=\frac{q}{p}\cdot \subc$, $\gcd(p,\frac{q}{p} \cdot \subc)=\subc$, and $\lcm(p,\frac{q}{p} \cdot \subc)=q$. If such a cycle  $C^1_{\psol}$ satisfies the previous conditions,
then it is called \textbf{feasible} and $\subc$ is said to be a \textbf{feasible divisor} of $q$. Following this idea, let $\mddD_{p,q}=\set{\elemD_1, \ldots, \elemD_e}$ be the set of the feasible divisors of $q$. Therefore, the basic equation admits 
at least one solution iff there exists a set of 
 non negative integers $\cmpv_1, \ldots , \cmpv_e$ such that $\sum_{i=1}^e \elemD_i \cdot \cmpv_i=n$. In that case, the solution corresponding to the tuple $\cmpv_1, \ldots , \cmpv_e$ is the sum of all those $C_{ \psol}^{\elemD_i \cdot \cmpv_i}$ with $\cmpv_i\neq 0$ and where $\psol=\frac{q}{p}\cdot \elemD_i$. 
\smallskip

We now describe a method based on \emph{Symmetry Breaking MDD} (SB-MDD)  enumerating the solutions of the considered basic equation. First of all, let us introduce the MDD $\sbmdd_{p,q,n}$ which is the  labelled digraph \structure{V,E,\lf} 
with vertices forming 
$V=\sum_{i=1}^\levelsb V_i$, 
where $\levelsb=\lfloor\frac{n}{\min \mddD_{p,q}}\rfloor+1$, 
$V_1=\{\rootnode\}$, $V_i$ is a multiset
of $\{1,\ldots,n-1\}$ for $i\in\set{2,\ldots,\levelsb-1}$, and, finally, $V_\levelsb=\{\ttnode\}$. 
For any node $\nodemdd\in V$, let $val(\nodemdd)=\nodemdd$
if $\nodemdd\ne\rootnode$\ and $\nodemdd\ne\ttnode$,
$val(\rootnode)=0$, and $val(\ttnode)=n$.

The structure is defined level by level as follows. For each level  $i\in\set{1,\ldots,\levelsb-2}$, for any $\nodemdd \in V_i$ and any $\nodemddbis \in \{1,\ldots,n-1\}$, it holds that $\nodemddbis\in V_{i+1}$ and $(\nodemdd, \nodemddbis)\in E$ iff $\nodemddbis-val(\nodemdd)\in \mddD_{p,q}$ and  $\nodemddbis \leq val(\ttnode)$. As far as the level $i=Z-1$ is concerned, for any $\nodemdd \in V_i$  it holds that  $(\nodemdd, \ttnode)\in E$ iff $val(\ttnode)-val(\nodemdd)\in \mddD_{p,q}$ and  $\nodemddbis \leq val(\ttnode)$. 
The labelling map $\lf: E\to \mddD_{p,q}$ associates any edge $(\nodemdd, \nodemddbis)\in E$ with the value $\lf(\nodemdd, \nodemddbis)=val(\nodemddbis)-val(\nodemdd)\in \mddD_{p,q}$.  

Once  $\sbmdd_{p,q,n}$ is built and reduced according to the p-Reduction from~\cite{perez2015efficient}, all the solutions of the considered basic equation can be computed.  Indeed, each solution corresponds to the sequence of the edge labels of a path from \rootnode\ to \ttnode consisting of possibly repeated values of $\mddD_{p,q}$ with sum equal to $n$. From such a sequence it is immediate to identify the above mentioned tuple $\cmpv_1, \ldots , \cmpv_e$ and then the corresponding solution. 

We stress that possible permutations of each of the above mentioned sequences can be provided by $\sbmdd_{p,q,n}$. In other words, distinct paths from \rootnode\ to \ttnode can lead to the same solution of the given basic equation. To reduce the size of such a MDD, during its construction and before the p-reduction, a symmetry breaking constraint can be imposed: for each node $\nodemdd\neq \ttnode$ the only allowed outgoing edges are those having a label which is less or equal to that of any of its incoming edges. In this way, any sequence of edge labels read on the paths of the structure turns out to be ordered and the size of the structure becomes smaller. The obtained MDD is called  SB-MDD, \ie, one which satisfies the symmetry breaking constraint. 
\begin{example}
Consider the basic equation $C^1_4 \odot X=C^{12}_{12}$. The set of divisors of $q=12$ (smaller or equal to $n=12$) is $\set{12,6,4,3,2,1}$. Thus, $\mddD_{p,q}=\set{4,2,1}$.
In fact, the following situations occur 
\begin{center}
    $\subc=12 \text{ and } \psol=36 \rightarrow gcd(4,36)\neq 12 \text{ and } lcm(4,36)\neq 12$\\
    $\subc=6 \text{ and } \psol=18 \rightarrow gcd(4,18)\neq 6 \text{ and } lcm(4,18)\neq 12$\\
    $\subc=4 \text{ and } \psol=12 \rightarrow gcd(4,12)=4 \text{ and } lcm(4,12)=12$\\
    $\subc=3 \text{ and } \psol=9 \rightarrow gcd(4,9)\neq 3 \text{ and } lcm(4,9)\neq 12$\\
    $\subc=2 \text{ and } \psol=6 \rightarrow gcd(4,6)=2 \text{ and } lcm(4,6)=12$\\
    $\subc=1 \text{ and } \psol=3 \rightarrow gcd(4,3)=1 \text{ and } lcm(4,3)=12$
\end{center}
Figure~\ref{fig:M4,12,12} shows the result of the reduction over $\sbmdd_{4,12,12}$. Solutions correspond to sequences of edge labels of paths from \rootnode\ to \ttnode.  These sequences form the following set:
\[ \{[4,4,4],[4,4,2,2],[4,4,2,1,1],[4,4,1,1,1,1],[4,2,2,2,2],[4,2,2,2,1,1],[4,2,2,1,1,1,1],[4,2,1,1,1,1,1,1],\]
\[ [4,1,1,1,1,1,1,1,1],[2,2,2,2,2,2],[2,2,2,2,2,1,1],[2,2,2,2,1,1,1,1],[2,2,2,1,1,1,1,1,1],\]
\[[2,2,1,1,1,1,1,1,1,1],[2,1,1,1,1,1,1,1,1,1,1],[1,1,1,1,1,1,1,1,1,1,1,1]\}.\]
\begin{figure}[htb]
  \begin{center}
	\begin{tikzpicture}[scale=.3,-latex,auto, semithick , state/.style ={circle,draw,minimum width=6mm},font=\footnotesize]
	\node[state,scale=.65] (r) at (0,0) {$r$};
	\node[state,scale=.65] (1=1) at (-9,-3) {$1$};
	\node[state,scale=.65] (1=2) at (-6,-3) {$2$};
	\node[state,scale=.65] (1=4) at (0,-3) {$4$};
	\node[state,scale=.65] (2=2) at (-10,-6) {$2$};
	\node[state,scale=.65] (2=3) at (-7,-6) {$3$};
	\node[state,scale=.65] (2=4) at (-5,-6) {$4$};
	\node[state,scale=.65] (2=5) at (-1,-6) {$5$};
	\node[state,scale=.65] (2=6) at (1,-6) {$6$};
	\node[state,scale=.65] (2=8) at (8,-6) {$8$};
	\node[state,scale=.65] (3=3) at (-10,-9) {$3$};
	\node[state,scale=.65] (3=4) at (-7,-9) {$4$};
	\node[state,scale=.65] (3=5) at (-5,-9) {$5$};
	\node[state,scale=.65] (3=6) at (-3,-9) {$6$};
	\node[state,scale=.65] (3=6=b) at (-1,-9) {$6$};
	\node[state,scale=.65] (3=7) at (1,-9) {$7$};
	\node[state,scale=.65] (3=8) at (3,-9) {$8$};
	\node[state,scale=.65] (3=9) at (7,-9) {$9$};
	\node[state,scale=.65] (3=10) at (9,-9) {$10$};
    \node[state,scale=.65] (4=4) at (-10,-12) {$4$};
    \node[state,scale=.65] (4=5) at (-7,-12) {$5$};
    \node[state,scale=.65] (4=6) at (-5,-12) {$6$};
    \node[state,scale=.65] (4=7) at (-3,-12) {$7$};
    \node[state,scale=.65] (4=8) at (-1,-12) {$8$};
    \node[state,scale=.65] (4=8=b) at (1,-12) {$8$};
    \node[state,scale=.65] (4=9) at (3,-12) {$9$};
    \node[state,scale=.65] (4=10) at (5,-12) {$10$};
    \node[state,scale=.65] (4=10=b) at (7,-12) {$10$};
    \node[state,scale=.65] (4=11) at (9,-12) {$11$};
    \node[state,scale=.65] (5=5) at (-10,-15) {$5$};
    \node[state,scale=.65] (5=6) at (-7,-15) {$6$};
    \node[state,scale=.65] (5=7) at (-5,-15) {$7$};
    \node[state,scale=.65] (5=8) at (-3,-15) {$8$};
    \node[state,scale=.65] (5=10) at (-1,-15) {$10$};
    \node[state,scale=.65] (5=9) at (1,-15) {$9$};
    \node[state,scale=.65] (5=10=b) at (3,-15) {$10$};
    \node[state,scale=.65] (5=11) at (6,-15) {$11$};
    \node[state,scale=.65] (6=6) at (-10,-18) {$6$};
    \node[state,scale=.65] (6=7) at (-7,-18) {$7$};
    \node[state,scale=.65] (6=8) at (-5,-18) {$8$};
    \node[state,scale=.65] (6=9) at (-3,-18) {$9$};
    \node[state,scale=.65] (6=10) at (1,-18) {$10$};
    \node[state,scale=.65] (6=11) at (3,-18) {$11$};
    \node[state,scale=.65] (7=7) at (-10,-21) {$7$};
    \node[state,scale=.65] (7=8) at (-7,-21) {$8$};
    \node[state,scale=.65] (7=9) at (-5,-21) {$9$};
    \node[state,scale=.65] (7=10) at (-3,-21) {$10$};
    \node[state,scale=.65] (7=11) at (1,-21) {$11$};
    \node[state,scale=.65] (8=8) at (-10,-24) {$8$};
    \node[state,scale=.65] (8=9) at (-7,-24) {$9$};
    \node[state,scale=.65] (8=10) at (-5,-24) {$10$};
    \node[state,scale=.65] (8=11) at (-3,-24) {$11$};
    \node[state,scale=.65] (9=9) at (-9,-27) {$9$};
    \node[state,scale=.65] (9=10) at (-6,-27) {$10$};
    \node[state,scale=.65] (9=11) at (-4,-27) {$11$};
    \node[state,scale=.65] (10=10) at (-8,-30) {$10$};
    \node[state,scale=.65] (10=11) at (-5,-30) {$11$};
    \node[state,scale=.65] (11=11) at (-6,-33) {$11$};
    \node[state,scale=.65] (tt) at (0,-36) {$tt$};
    
    \draw [->]  (r) to node[above,scale=.65] {$1$} (1=1);
    \draw [->]  (r) to node[below,scale=.65] {$2$} (1=2);
    \draw [->]  (r) to node[right,scale=.65] {$4$} (1=4);
    \draw [->]  (1=1) to node[left,scale=.65] {$1$} (2=2);
    \draw [->]  (1=2) to node[left,scale=.65] {$1$} (2=3);
    \draw [->]  (1=2) to node[right,scale=.65] {$2$} (2=4);
    \draw [->]  (1=4) to node[left,scale=.65] {$1$} (2=5);
    \draw [->]  (1=4) to node[right,scale=.65] {$2$} (2=6);
    \draw [->]  (1=4) to node[above,scale=.65] {$4$} (2=8);
    \draw [->]  (2=2) to node[left,scale=.65] {$1$} (3=3);
    \draw [->]  (2=3) to node[left,scale=.65] {$1$} (3=4);
    \draw [->]  (2=4) to node[left,scale=.65] {$1$} (3=5);
    \draw [->]  (2=4) to node[left,scale=.65] {$2$} (3=6);
    \draw [->]  (2=5) to node[left,scale=.65] {$1$} (3=6=b);
    \draw [->]  (2=6) to node[left,scale=.65] {$1$} (3=7);
    \draw [->]  (2=6) to node[right,scale=.65] {$2$} (3=8);
    \draw [->]  (2=8) to node[left,scale=.65] {$1$} (3=9);
    \draw [->]  (2=8) to node[right,scale=.65] {$2$} (3=10);
    \path[->] (2=8.0)  edge [bend left=50] node {\scalebox{0.65}{$4$}} (tt);
    \draw [->]  (3=3) to node[left,scale=.65] {$1$} (4=4);
    \draw [->]  (3=4) to node[left,scale=.65] {$1$} (4=5);
    \draw [->]  (3=5) to node[left,scale=.65] {$1$} (4=6);
    \draw [->]  (3=6) to node[left,scale=.65] {$1$} (4=7);
    \draw [->]  (3=6) to node[right,scale=.65] {$2$} (4=8);
    \draw [->, bend right]  (3=6=b) to node[below,scale=.65] {$1$} (4=7.65);
    \draw [->]  (3=7) to node[left,scale=.65] {$1$} (4=8=b);
    \draw [->]  (3=8) to node[left,scale=.65] {$1$} (4=9);
    \draw [->]  (3=8) to node[right,scale=.65] {$2$} (4=10);
    \draw [->]  (3=9) to node[left,scale=.65] {$1$} (4=10=b);
    \draw [->]  (3=10) to node[left,scale=.65] {$1$} (4=11);
    \path[->] (3=10.0)  edge [bend left=30] node {\scalebox{0.65}{$2$}} (tt);
    \draw [->]  (4=4) to node[left,scale=.65] {$1$} (5=5);
    \draw [->]  (4=5) to node[left,scale=.65] {$1$} (5=6);
    \draw [->]  (4=6) to node[left,scale=.65] {$1$} (5=7);
    \draw [->]  (4=7) to node[left,scale=.65] {$1$} (5=8);
    \draw [->]  (4=8) to node[left,scale=.65] {$2$} (5=10);
    \draw [->]  (4=8) to node[right,scale=.65] {$1$} (5=9);
    \draw [->]  (4=8=b) to node[right,scale=.65] {$1$} (5=9);
    \draw [->]  (4=9) to node[left,scale=.65] {$1$} (5=10=b);
    \draw [->]  (4=10) to node[right,scale=.65] {$1$} (5=11);
    \draw [->]  (4=10=b) to node[right,scale=.65] {$1$} (5=11);
    \path[->] (4=10.260)  edge [bend left=10] node {\scalebox{0.65}{$2$}} (tt);
    \path[->] (4=11.270)  edge [bend left=20] node {\scalebox{0.65}{$1$}} (tt);
    \draw [->]  (5=5) to node[left,scale=.65] {$1$} (6=6);
    \draw [->]  (5=6) to node[left,scale=.65] {$1$} (6=7);
    \draw [->]  (5=7) to node[left,scale=.65] {$1$} (6=8);
    \draw [->]  (5=8) to node[left,scale=.65] {$1$} (6=9);
    \draw [->]  (5=10) to node[left,scale=.65] {$2$} (tt);
    \draw [->]  (5=9) to node[left,scale=.65] {$1$} (6=10);
    \draw [->]  (5=10) to node[right,scale=.65] {$1$} (6=11);
    \draw [->]  (5=10=b) to node[right,scale=.65] {$1$} (6=11);
    \path[->] (5=11.270)  edge [bend left=15] node {\scalebox{0.65}{$1$}} (tt);
    \draw [->]  (6=6) to node[left,scale=.65] {$1$} (7=7);
    \draw [->]  (6=7) to node[left,scale=.65] {$1$} (7=8);
    \draw [->]  (6=8) to node[left,scale=.65] {$1$} (7=9);
    \draw [->]  (6=9) to node[left,scale=.65] {$1$} (7=10);
    \draw [->]  (6=10) to node[left,scale=.65] {$1$} (7=11);
    \path[->] (6=11.270)  edge [bend left=10,left] node {\scalebox{0.65}{$1$}} (tt);
    \draw [->]  (7=7) to node[left,scale=.65] {$1$} (8=8);
    \draw [->]  (7=8) to node[left,scale=.65] {$1$} (8=9);
    \draw [->]  (7=9) to node[left,scale=.65] {$1$} (8=10);
    \draw [->]  (7=10) to node[left,scale=.65] {$1$} (8=11);
    \path[->] (7=11.270)  edge [] node {\scalebox{0.65}{$1$}} (tt);
    \draw [->]  (8=8) to node[left,scale=.65] {$1$} (9=9);
    \draw [->]  (8=9) to node[left,scale=.65] {$1$} (9=10);
    \draw [->]  (8=10) to node[left,scale=.65] {$1$} (9=11);
    \draw [->]  (8=11) to node[left,scale=.65] {$1$} (tt);
    \draw [->]  (9=9) to node[left,scale=.65] {$1$} (10=10);
    \draw [->]  (9=10) to node[left,scale=.65] {$1$} (10=11);
    \draw [->]  (9=11) to node[left,scale=.65] {$1$} (tt);
    \draw [->]  (10=10) to node[left,scale=.65] {$1$} (11=11);
    \draw [->]  (10=11) to node[left,scale=.65] {$1$} (tt);
    \draw [->]  (11=11) to node[left,scale=.65] {$1$} (tt);

	\end{tikzpicture}
  \end{center}
	\caption{The reduced SB-MDD representing all the solutions of $C^1_4 \odot X=C^{12}_{12}$ .}
	\label{fig:M4,12,12}
\end{figure}
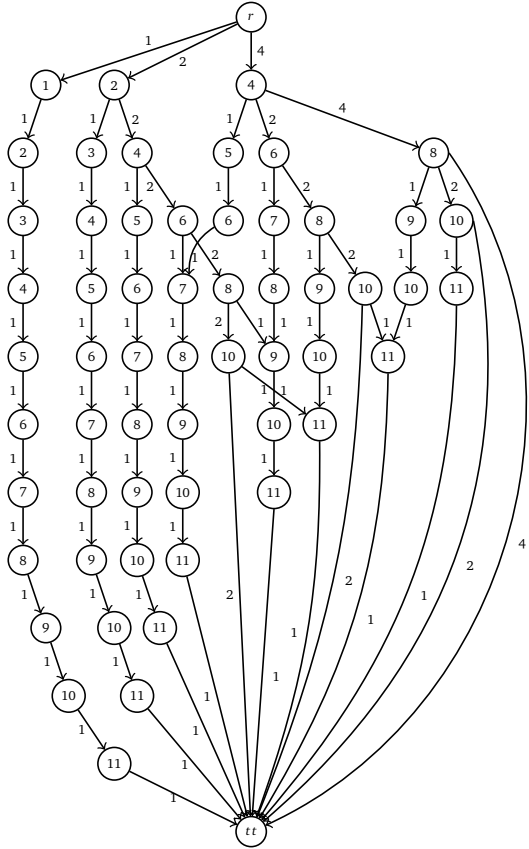
Each element $r$ of a sequence belongs to $\mddD_{p,q}=\set{4,2,1}$ and it corresponds to a cycle of length $\psol=\frac{q}{p} \cdot \subc$ of the solution represented by that sequence. As an example, the sequence $[4,4,2,1,1]$ gives rise to $2$ cycles of length $\psol=\frac{q}{p} \cdot 4=12$, $1$ cycle of length $\psol=\frac{q}{p} \cdot 2=6$, and $2$ cycles of length $\psol=\frac{q}{p} \cdot 1=3$, \ie, the solution $C^2_{12}\oplus C^1_6\oplus C^2_3$.
\[ \{C^3_{12},C^2_{12}\oplus C^2_6,C^2_{12}\oplus C^1_6\oplus C^2_3,C^2_{12}\oplus C^4_3,C^1_{12}\oplus C^4_6,C^1_{12}\oplus C^3_6\oplus C^2_3,C^1_{12}\oplus C^2_6\oplus C^4_3,C^1_{12}\oplus C^1_6\oplus C^6_3,\]
\[ C^1_{12}\oplus C^8_3,C^6_6,C^5_6\oplus C^2_3,C^4_6\oplus C^4_3,C^3_6\oplus C^6_3,C^2_6\oplus C^8_3,C^1_6\oplus C^{10}_3,C^{12}_3\}.\]
\end{example}
The method based on the above described SB-MDD also establishes the instances of equations without solutions via the following criteria:
\begin{itemize}
    \item if $p$ cannot divide $q$;
    \item if $\mddD_{p,q}$ is the empty set;
    \item if, after the reduction process, no valid paths from \rootnode\ to \ttnode\ remain in the SB-MDD structure.
\end{itemize}
The following example just illustrates how the method establishes whether an instance of a basic equation 
has no solutions.   
\begin{example}
Consider the equation $C^1_2 \odot X=C^{5}_{4}$. The set of divisors of $q$ (smaller or equal to $n$) is $\set{4,2,1}$. Thus, $\mddD_{p,q}=\set{2}$.
Indeed, the following situations occur
\begin{center}
    $\subc=4 \text{ and } \psol=8 \rightarrow gcd(2,8)\neq 4 \text{ and } lcm(2,8)\neq 4$\\
    $\subc=2 \text{ and } \psol=4 \rightarrow gcd(2,4)=2 \text{ and } lcm(2,4)=4$\\
    $\subc=1 \text{ and } \psol=2 \rightarrow gcd(2,2)\neq 1 \text{ and } lcm(2,2)\neq 4$
\end{center}
Figure~\ref{fig:M5,4,2} shows $\sbmdd_{5,4,2}$ before the reduction procedure. The red part is deleted 
when the reduction phase is performed. The SB-MDD has no paths from the \rootnode\ to \ttnode\ node, and, hence, the equation has no solutions.
\begin{figure}[t]
  \begin{center}
	\begin{tikzpicture}[scale=.3,-latex,auto, semithick , state/.style ={circle,draw,minimum width=6mm},font=\footnotesize]
	\node[state,scale=.65] (r) at (0,0) {$r$};
	\node[state,scale=.65,red] (1=1) at (0,-3) {$2$};
	\node[state,scale=.65,red] (2=2) at (0,-6) {$2$};
	\node[state,scale=.65] (tt) at (0,-9) {$tt$};
    
    \draw [->,red]  (r) to node[right,scale=.65,red] {$2$} (1=1);
    \draw [->,red]  (1=1) to node[right,scale=.65,red] {$2$} (2=2);

	\end{tikzpicture}
  \end{center}
	\caption{The SB-MDD (before reduction) representing all the solutions of $C^1_2 \odot X=C^{5}_{4}$. The red part is deleted by the pReduction procedure.}
	\label{fig:M5,4,2}
\end{figure}
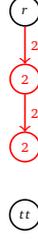
\end{example}
Experiments show how this method can achieve interesting performances in time and memory \cite{riva2021mdd}. 

\subsection{Contraction steps} \label{cssec}
We now present how all feasible Systems~\eqref{systelcontr} can be first generated starting from Equation~\eqref{eq:simplified} and then solved. Since Systems~\eqref{systelcontr} 
may lead to basic equations without solutions and the same basic equation may be reached several times as far as distinct systems are considered, 
first of all the basic equations that can be involved have to be individuated and among them only the \textbf{necessary} ones, \ie, those admitting a solution, have to be solved just once. 

The identification of all the involved basic equations consists in considering all the SB-MDD $\sbmdd_{p_{zi},p_{j},n/n_{zi}}$ defined by varying $z \in \{1,\ldots ,\mons\}$, $i \in \{1,\ldots ,\nPSimp_z\}$, $j \in \{1,\ldots,\Pright\}$, and $n \in \{1, \ldots, n_j\}$. Then, those SB-MDD  corresponding to necessary basic equations are computed, \ie, all the necessary basic equations are solved.  
\medskip

We now describe an MDD able to generate all  feasible Systems~\eqref{systelcontr}. Such an MDD is $CS=CS_1 \times \ldots \times CS_{\Pright}$, \ie, the Cartesian product of $\Pright$ MDD, where each $CS_j$ aims at providing, according to the set of the necessary equations,  all the feasible ways by which the monomials of Equation~\eqref{eq:simplified} can concur to form the $n_j$ cycles of length $p_j$ of the known term $\mathring{\rterm}$. Clearly, by the Stars and Bars method those ways are at most  $\binom{n_j+\nPSimp-1}{\nPSimp-1}$ and, hence, there are at most $\prod^{\Pright}_{j=1} \binom{n_j+\nPSimp-1}{\nPSimp-1}$ feasible Systems~\eqref{systelcontr}. Furthermore, by definition, the whole MDD $CS$ will provide all the feasible ways by which all the cycles of $\mathring{\rterm}$ can be formed.

Each $CS_{j}$ is a labelled digraph \structure{V_{j},E_{j},\lf_{j}} in which there are $\mons\cdot \nPSimp_z$ levels, one for each monomial $C^{n_{zi}}_{p_{zi}} \odot X_z$ from the left-hand side of Equation~\eqref{eq:simplified}, besides the level containing the only terminal node $\ttnode$.  The vertex set is  $V_{j}=(\sum_{z\in\{1,\ldots,\mons\}\;\;\;i\in\{1,\ldots,\nPSimp_z\}}^{} V_{j,zi} )+ V_{j,(\mons+1)1}$ where  
$V_{j,11}=\set{\rootnode}$, $V_{j,(\mons+1)1}=\{\ttnode\}$, and for each pair $(z,i)$ with $z\in\{1,\ldots,\mons\}$ and $i\in\{1,\ldots,\nPSimp_z\}$ the set  $V_{j,zi}\subseteq \{0,\ldots,n_j\}$ of the vertexes of the level $(z,i)$ will be defined in the sequel. Indeed, the graph is built level by level. Moreover, for any node $\nodemdd\in V_j$, let $val(\nodemdd)=\nodemdd$
if $\nodemdd\ne\rootnode$\ and $\nodemdd\ne\ttnode$,
while $val(\rootnode)=0$ and $val(\ttnode)=n_j$.
To define the edges outgoing from the vertexes of any level along with the corresponding label and then the vertexes of the next level too, first of all we associate each level $(z,i)$ with the set $\mddD_{p_{zi},p_j}=\{\elemD \in \N \mid 1\leq \elemD\leq n_j \text{ and }\sbmdd_{p_{zi},p_j,\elemD/n_{zi}} \text{ is defined by a necessary equation}\}\cup \{0\}$ of 
the labels of the edges outgoing from the vertexes of that level. Now, for each level $(z,i)$ with $z\neq \mons$ and $i\neq \nPSimp_z$, for any vertex $\nodemdd\in V_{j,zi}$ and any $\nodemddbis\in \{0,\ldots,n_j\}$, it holds that  
\begin{itemize}
\item $\nodemddbis\in V_{j,z(i+1)}$  and  $(\nodemdd,\nodemddbis)\in E_j$ iff 
$\nodemddbis-val(\nodemdd)\in \mddD_{p_{zi},p_j}$ and $\nodemddbis \leq val(\ttnode)$, whenever $i< \nPSimp_z$;
\item $\nodemddbis\in V_{j,(z+1)1}$ and  $(\nodemdd,\nodemddbis)\in E_j$ iff 
$\nodemddbis-val(\nodemdd)\in \mddD_{p_{zi},p_j}$ and $\nodemddbis \leq val(\ttnode)$, whenever  $i= \nPSimp_z$.
\end{itemize}
Concerning the level $(\mons, \nPSimp_z)$, for any vertex $\nodemdd\in V_{j,\mons\nPSimp_z}$ it holds that $(\nodemdd,\ttnode)\in E_j$ iff $val(\ttnode)-val(\nodemdd)\in \mddD_{p_{\mons \nPSimp_z},p_j}$. Every edge  $(\nodemdd,\nodemddbis)\in E_j$ is associated with the label $\lf_j(\nodemdd,\nodemddbis)=val(\nodemddbis)-val(\nodemdd)\in \mddD_{p_{zi},p_j}$, where $(z,i)$ is such that $\nodemdd\in V_{j,zi}$. In this way, the labelling map $\lf_{j}\colon E_j\to \bigcup_{z\in\{1,\ldots,\mons\}\;\;\;i\in\{1,\ldots,\nPSimp_\mons\}}^{} \mddD_{p_{zi},p_j}$ has been defined too.

We stress that any edge outgoing from vertexes of the level $(z,i)$ represents the cycles of length $p_j$ that the monomial $C^{n_{zi}}_{p_{zi}} \odot X_z$ can contribute to form together with the monomials corresponding to the other edges encountered on a same path from \rootnode to \ttnode. The label of the edge is just the number $n^{zi}_{j}$ of those cycles and the sum of all the labels of the edges in any path from \rootnode to \ttnode is just the number $n_j$ cycles of length $p_j$ to be formed by the monomials $C^{n_{zi}}_{p_{zi}} \odot X_z$ of the left-hand side of Equation~\eqref{eq:simplified}. The value $val(\alpha)$ associated with a node $\alpha$ of a path from \rootnode to \ttnode is the partial result of that sum, \ie, the number of cycles of length $p_j$  formed by the monomials encountered on the subpath from \rootnode to $\alpha$. 

At this point, the MDD $CS$ is built and, according to the definition of cartesian product of MDDs, the involved MDDs are stacked on top of each other in such a way that each $CS_j$ turns out to be on top of  $CS_{j+1}$ and the terminal node of $CS_j$ is collapsed with the root of $CS_{j+1}$. Any path from the root to the terminal node of $CS$ represents a possible way by which the monomials  $C^{n_{zi}}_{p_{zi}} \odot X_z$ of the left-hand side of Equation~\eqref{eq:simplified} can concur to form all the cycles of $\mathring{\rterm}$, or, in other words, it corresponds to a possible solution of Equation~\eqref{eq:simplified}. In particular, since for each pair $(z,i)$ a level $(z,i)$ appears in every $CS_j$,  the set of the $\Pright$ edges in any of the above mentioned paths of $CS$, each of them outgoing from vertexes of the same level $(z,i)$ in one $CS_j$, 
defines a feasible way of solving the equation from System~\eqref{systelcontr}
\[C^{n_{zi}}_{p_{zi}} \odot X_z = \bigoplus\limits_{j=1}^{\Pright} C_{p_{j}}^{n^{zi}_{j}}\enspace,
\]
\ie, a way by which the monomial $C^{n_{zi}}_{p_{zi}} \odot X_z$ gives rise at the same time to $n^{zi}_{1}$ cycles of length $p_{1}$, $n^{zi}_{2}$ cycles of length $p_{2}$, \ldots, and $n^{zi}_{\Pright}$ cycles of length $p_{\Pright}$.  Therefore, all the monomials encountered in a path of $CS$ contribute to form a possibly feasible System~\eqref{systelcontr}. 
\begin{example}\label{cs_ex}
Consider the equation: \[C_4^1\odot X_1 \oplus  C_2^1 \odot X_2=C_2^4\oplus C_4^4 \oplus C_6^7 \oplus C_{12}^7  
.\] 
There are $44$ distinct basic equations and among them $27$ equations are necessary. Indeed, besides the basic equations defined by $p=4$ and $q\in\{2,6\}$, the following ones have no solution: $C_2^1 \odot X_2=C_4^1$, $C_2^1 \odot X_2=C_4^3$, $C_2^1 \odot X_2 =C_{12}^1$, $C_2^1 \odot X_2=C_{12}^3$,
$C_2^1 \odot X_2=C_{12}^5$, and $C_2^1 \odot X_2=C_{12}^7$.\\
To illustrate one $CS_j$, let us consider $j=2$, or, in other words, the MDD providing all the possible ways by which the two monomials of the given equation can concur to form $C_4^4$. Thus, $CS_2$ has $2$ levels, one for each monomial. Any edge outgoing from a level represents the cycles of length $4$, along the number of them, that the monomial corresponding to that level can contribute to form. The first level, corresponding to the monomial $C_4^1\odot X_1$, only contains
the node \rootnode. According to the necessary equations defined $p=4$ and $q=4$, the first monomial is able by itself to form $n^{11}_{2}$ cycles of length $4$ where  $n^{11}_{2}\in\{1,2,3,4\}$.  
Regarding the second monomial, it is able by itself to form either $n^{21}_{2}=2$ or $n^{21}_{2}=4$ cycles of length $4$. 
As Figure \ref{fig:cs2} shows, the MDD $CS_2$ also represents the cases $n^{11}_{2}=0$ and/or $n^{21}_{2}=0$, \ie, where at least one of the two monomials does not contribute to the generation of such cycles at all. Any path from \rootnode to \ttnode provides a feasible way by which the two monomials concur to form $n_2=4$ cycles of length $p_2=4$. 
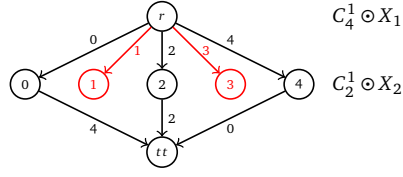
\begin{figure}[t]
  \begin{center}
	\begin{tikzpicture}[scale=.3,-latex,auto, semithick , state/.style ={circle,draw,minimum width=6mm},font=\footnotesize]
	\node[state,scale=.65] (r) at (0,0) {$r$};
	\node[state,scale=.65] (2) at (0,-3) {$2$};
	\node[state,scale=.65,red] (3) at (3,-3) {$3$};
	\node[state,scale=.65] (4) at (6,-3) {$4$};
	\node[state,scale=.65,red] (1) at (-3,-3) {$1$};
	\node[state,scale=.65] (0) at (-6,-3) {$0$};
	\node[state,scale=.65] (tt) at (0,-6) {$tt$};
	\node[state,white] (m1) at (9,0) {\textcolor{black}{$C_4^1\odot X_1$}};
	\node[state,white] (m1) at (9,-3) {\textcolor{black}{$C_2^1 \odot X_2$}};
    \draw [->]  (r) to node[above,scale=.65] {$0$} (0);
    \draw [->,red]  (r) to node[right,scale=.65,red] {$1$} (1);
    \draw [->]  (r) to node[right,scale=.65] {$2$} (2);
    \draw [->,red]  (r) to node[right,scale=.65,red] {$3$} (3);
    \draw [->]  (r) to node[above,scale=.65] {$4$} (4);
    \draw [->]  (0) to node[below,scale=.65] {$4$} (tt);
    \draw [->]  (2) to node[right,scale=.65] {$2$} (tt);
    \draw [->]  (4) to node[below,scale=.65] {$0$} (tt);

	\end{tikzpicture}
  \end{center}
	\caption{The MDD $CS_2$ represents all the possible ways by which, according to the set of necessary equations, the two monomials can concur to form $C^4_4$. The red part is deleted by the pReduction procedure. The value $val(\alpha)$ associated with each node $\alpha$ is also reported.}
	\label{fig:cs2}
\end{figure}
Figure \ref{fig:cs} illustrates the MDD $CS=CS_1\times CS_2\times CS_3\times CS_4$ associated with the given equation and obtained by stacking each $CS_j$ on top of $CS_{j+1}$. Any path from the root to the terminal node of $CS$ represents a possible way by which the monomials of the left-hand side of the given equation 
can concur to form all the cycles of its known term. 
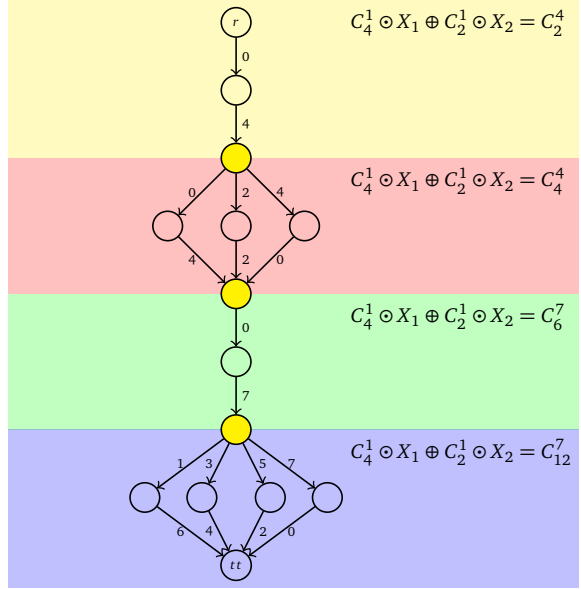
\begin{figure}[t]
  \begin{center}
	\begin{tikzpicture}[scale=.3,-latex,auto, semithick , state/.style ={circle,draw,minimum width=6mm},font=\footnotesize]
	\path [fill=yellow, nearly transparent] (-10,1) rectangle (15,-6);
    \path [fill=red, nearly transparent] (-10,-6) rectangle (15,-12);
    \path [fill=green, nearly transparent] (-10,-12) rectangle (15,-18);
    \path [fill=blue, nearly transparent] (-10,-18) rectangle (15,-25);
	\node[state,scale=.65] (r) at (0,0) {$r$};
	\node[text width=3cm] at (10,0) 
    {$C_4^1\odot X_1 \oplus  C_2^1 \odot X_2=C_2^4$};
    \node[text width=3cm] at (10,-7) 
    {$C_4^1\odot X_1 \oplus  C_2^1 \odot X_2=C_4^4$};
    \node[text width=3cm] at (10,-13) 
    {$C_4^1\odot X_1 \oplus  C_2^1 \odot X_2=C_{6}^7$};
	\node[text width=3cm] at (10,-19) 
    {$C_4^1\odot X_1 \oplus  C_2^1 \odot X_2=C_{12}^7$};
	\node[state,scale=.65] (1) at (0,-3) {};
	\node[state,scale=.65,fill=yellow] (tt1) at (0,-6) {};
    \node[state,scale=.65] (2) at (0,-9) {};
    \node[state,scale=.65] (2u) at (3,-9) {};
    \node[state,scale=.65] (2d) at (-3,-9) {};
    \node[state,scale=.65,fill=yellow] (tt2) at (0,-12) {};
    \node[state,scale=.65] (3) at (0,-15) {};
    \node[state,scale=.65,fill=yellow] (tt3) at (0,-18) {};
    \node[state,scale=.65] (4uu) at (4,-21) {};
    \node[state,scale=.65] (4u) at (1.5,-21) {};
    \node[state,scale=.65] (4d) at (-1.5,-21) {};
    \node[state,scale=.65] (4dd) at (-4,-21) {};
    \node[state,scale=.65] (tt4) at (0,-24) {$tt$};
    \draw [->]  (r) to node[right,scale=.65] {$0$} (1);
    \draw [->]  (1) to node[right,scale=.65] {$4$} (tt1);
    \draw [->]  (tt1) to node[right,scale=.65] {$2$} (2);
    \draw [->]  (tt1) to node[right,scale=.65] {$4$} (2u);
    \draw [->]  (tt1) to node[left,scale=.65] {$0$} (2d);
    \draw [->]  (2u) to node[right,scale=.65] {$0$} (tt2);
    \draw [->]  (2) to node[right,scale=.65] {$2$} (tt2);
    \draw [->]  (2d) to node[left,scale=.65] {$4$} (tt2);
    \draw [->]  (tt2) to node[right,scale=.65] {$0$} (3);
    \draw [->]  (3) to node[right,scale=.65] {$7$} (tt3);
    \draw [->]  (tt3) to node[right,scale=.65] {$7$} (4uu);
    \draw [->]  (tt3) to node[right,scale=.65] {$5$} (4u);
    \draw [->]  (tt3) to node[left,scale=.65] {$3$} (4d);
    \draw [->]  (tt3) to node[left,scale=.65] {$1$} (4dd);
    \draw [->]  (4uu) to node[right,scale=.65] {$0$} (tt4);
    \draw [->]  (4u) to node[right,scale=.65] {$2$} (tt4);
    \draw [->]  (4d) to node[left,scale=.65] {$4$} (tt4);
    \draw [->]  (4dd) to node[left,scale=.65] {$6$} (tt4);

	\end{tikzpicture}
  \end{center}
	\caption{The MDD $CS$ represents all the feasible ways by which, according to the set of necessary equations, the monomials of the equation from Example~\ref{cs_ex} can concur to form its right-hand side. According to the cartesian product of MDD, the yellow nodes are at the same time the \ttnode node of a $CS_j$ and the \rootnode node of  $CS_{j+1}$. The four MDDs are depicted by  different colours (the red MDD corresponds to that from Figure~\ref{fig:cs2}). In each $CS_j$ the first (resp., second) level corresponds to the monomial $C_4^1\odot X_1$ (resp., $C_2^1 \odot X_2$). The values $val(\alpha)$ associated to nodes are omitted for simplicity.}
	\label{fig:cs}
\end{figure}
\end{example}

Now, solving any equation from a   System~\eqref{systelcontr} means computing the cartesian product among the solutions of the $\Pright$  equations in~\eqref{eq:tanteq}. Since each of them can be equivalently rewritten as a basic equation, this can be performed by computing the cartesian product of the SB-MDD, each providing the solutions of the involved basic equation. As usual, such a cartesian product, that we name \textbf{SB-Cartesian MDD}, is obtained by stacking the SB-MDDs on top of each other. In this way, one can get the values of the $X_z$  satisfying  any equation 
of the System~\eqref{systelcontr}
defined by a path of $CS$. We stress that  an SB-Cartesian MDD is not a SB-MDD. 
In particular, although it is satisfied by each of its component SB-MDD, the order constraint among the edge labels of any path from the root to the terminal node of an SB-Cartesian MDD does not hold.

To provide the solutions of a System~\eqref{systelcontr}, for each $X_z$ the intersection among the solutions of all the equations involving the same variable $X_z$ is required. Then, once the values of the $x_z$ will have been computed  starting from the values of $X_z$  
by means of the  algorithm presented in  Section~\ref{roots}, a further intersection of the sets of values of a same $x_z$ arisen from distinct $X_z$ (if any)  will be performed. Indeed, there can be equations in distinct variables that however are (distinct) powers of the a same variable $x_z$. We now deal with the first  mentioned intersection (the second one is standard and it can be performed in such a way that  only one root of the variables $X_z$ that are powers of a same $x_z$ is computed).

According to the current state of the art, 
there exists an algorithm that, starting from two MDD, possibly two SB-MDD, each of them providing the solutions of an equation, builds a new MDD able to compute the intersection between the solutions of the two equations. Essentially, each node in the new structure corresponds to two nodes, one from each MDD, and the procedure recreates an outgoing edge in the structure if it is common to both the MDDs. For more details, we refer the reader to~\cite{perez2015efficient} and~\cite{bergman2014mdd}.

Nevertheless, such an algorithm can not be used if SB-Cartesian MDDs are involved, as  it happens instead in our scenario,  unless each  monomial gives rise to cycles of a unique length, \ie, the solutions of each corresponding equation are computed by a SB-MDD.  
Indeed, the result of the above mentioned algorithm  depends on the order by which the SB-Cartesian MDDs are considered when  the intersection is performed. 
In \cite{riva2021mdd}, a new algorithm performing the intersection  has been proposed in such a way that it properly works independently of that order. Let us recall its underlying idea. 
\smallskip

The algorithm starts to compute the intersection among the solutions of equations provided by all the SB-MDDs, if any. If it is not empty, such an intersection consists of a set of candidate solutions that form the so-called initial guess. Otherwise, the initial guess is the set of the solutions provided by one of the SB-Cartesian MDDs. The current set of candidate solutions which at the beginning is just the initial guess is updated by means of the intersection between itself and the set of the solutions provided by one of the SB-Cartesian MDDs that have not yet been considered. Any intersection essentially consists in visiting the chosen SB-Cartesian MDD $CS$ to establish whether a candidate solution is provided by one among the SB-MDD components $CS_j$ of $CS$.  If this does not happen, it is removed from the set of candidate solutions.

\subsection{Roots of DDS} \label{roots}
We now deal with the problem of retrieving the value of each DDS  $\mathring{x_z}$ once the DDS $X_z$ have been computed. Since each $X_z$ is the $\esp_z$-th power of $\mathring{x_z}$, we are going to introduce the concept of $\esp$-th root in the semiring of DDS and provide an algorithm for computing the $\esp$-th roots of the a-abstractions of DDS. 

First of all, let us formally define the notion of
$w$-root of a general DDS.
\begin{definition}
Let $\esp\geq 2$ be a natural number. The $\esp$-th root of a DDS $\dds$ is a DDS having $\esp$-th power equal to $\dds$.
\end{definition}
Clearly, the a-abstraction of the  $\esp$-th root of a DDS is the $\esp$-th root of the a-abstraction of that system.
The goal is now to compute the  $\esp$-th root of the a-abstraction 
of any DDS. Namely, for any given a-abstraction
$$C^{\possn_1}_{\possp_1} \oplus \ldots \oplus C^{\possn_h}_{\possp_h}\enspace,$$
with $0<\possp_1<\possp_2< \ldots<\possp_h$, 
we want to solve the equation 
\begin{equation}\label{eq:root}
    \mathring{x}^{\esp}=C^{\possn_1}_{\possp_1} \oplus  \ldots \oplus C^{\possn_h}_{\possp_h}\enspace,
\end{equation}
where the unknown is expressed as
\begin{equation*}
\mathring{x}=C^{\radn_1}_{\radp_1} \oplus  \ldots \oplus C^{\radn_l}_{\radp_l},
\end{equation*}
for some naturals $l, \radp_1,\ldots, \radp_l, \radn_1,\ldots, \radn_l$ 
$\radp_1<\ldots< \radp_l$, and $\radn_1,\ldots, \radn_l$
to be determined.

\paragraph{Assumption} From now on, without loss of generality, we will assume $\radp_1<\ldots< \radp_l$, and $\possp_1<\ldots< \possp_h$.
\medskip

Since providing a closed formula for $\mathring{x}$ is essentially unfeasible, we are going to compute the sets  $C^{\radn_i}_{\radp_i}$ one by one starting from $i=1$. 
Such a computation will be iteratively performed  by considering the generation of the sets  $C^{\possn_j}_{\possp_j}$ by carrying out the $w$-th power of the sum of sets $C^{\radn_i}_{\radp_i}$.
\begin{proposition}\label{prop:base}
For any natural $l\geq 2$, if $\mathring{x}=C^{\radn_1}_{\radp_1} \oplus  \ldots \oplus C^{\radn_l}_{\radp_l}$ is a solution of the equation $\mathring{x}^{\esp}=C^{\possn_1}_{\possp_1} \oplus  \ldots \oplus C^{\possn_h}_{\possp_h}$, then  all the following facts hold:
\begin{enumerate}[(i)]
    \item  $l\leq h$ and $\{\radp_1, \ldots, \radp_l\}\subseteq \{\possp_1, \ldots, \possp_h\}$
    \item $\radp_1=\possp_1$ and $\radp_2=\possp_2$;
    \item $\radn_1=\sqrt[\leftroot{-2}\uproot{2}\esp]{\possn_1\over{\possp_1^{\esp-1}}}\in\N$; 
   
    \item 
$
\begin{aligned}
  \radn_2  =
  \begin{cases}
  \frac{\sqrt[\leftroot{-2}\uproot{2}\esp]{\possp_2\possn_2 + \possp_1\possn_1} -  \sqrt[\leftroot{-2}\uproot{2}\esp]{\possp_1\possn_1}}{\possp_2} \in\N, & \text{if } \lcm(\possp_1,\possp_2)=\possp_2,\\
  \sqrt[\leftroot{-2}\uproot{2}\esp]{\possn_2\over{\possp_2^{\esp-1}}}\in\N, & \text{otherwise.}
  \end{cases}
\end{aligned}
$

\end{enumerate} 
\end{proposition}
\begin{proof}\mbox{}\\
\indent$(i)$: According to Proposition~\ref{general_newt}, for each $i\in\{1,\ldots, l\}$, a set $C^{\possn}_{\lcmK_\nP}$ with $\lcmK_\nP=\radp_i$ appears in $\mathring{x}^{\esp}$ when the tuple $(k_1, \ldots, k_l)$ with $k_i=w$ and $k_{i'}=0$ for $i'\neq i$ is involved in the sum.  
Hence,  $\set{\radp_1, \ldots, \radp_l}\subseteq \set{\possp_1, \ldots, \possp_h}$ and $l \leq h$.

$(ii)$: Since $\radp_1$ is the smallest value among all possible $\lcm$ $\lcmK_\nP$ from  Proposition~\ref{general_newt} and   $\possp_1$ is the smallest among the lengths $\possp_1, \ldots, \possp_h$ of the cycles to be generated when the $w$-th  power of $\mathring{x}$  is performed, it must necessarily hold that $\radp_1=\possp_1$ in order that, in particular, cycles of length $\possp_1$ are generated. Moreover, since $\radp_2$ and $\possp_2$ follow in ascending order $\radp_1$ and $\possp_1$, respectively, and $\radp_2$ is also the successor of $\radp_1$ among all the above mentioned $\lcm$, it must also hold that $\radp_2=\possp_2$ in order that cycles of length $\possp_2$ are generated too.

$(iii)$:
Actually, it holds that $(C^{\radn_1}_{\possp_1})^\esp=C^{\possn_1}_{\possp_1}$, which,  by Corollary~\ref{lem:s1}, is equivalent to  $C^{\possp_1^{\esp-1}\radn_1^\esp}_{\possp_1}=C^{\possn_1}_{\possp_1}$. This implies that  $\possp_1^{\esp-1}\radn_1^\esp=\possn_1$ and, hence,  $\radn_1=\sqrt[\leftroot{-2}\uproot{2}\esp]{\possn_1\over{\possp_1^{\esp-1}}}$.

$(iv)$: if 
$\lcm(\possp_1,\possp_2)>\possp_2$, when computing the $w$-th power of $\mathring{x}$, by Lemma~\ref{lem:s1},  $C^{\radn_1}_{\radp_1}$ does not contribute to form $C^{\possn_2}_{\possp_2}$ and, necessarily, it holds that $(C^{\radn_2}_{\possp_2})^\esp=C^{\possp_2^{\esp-1}\radn_2^\esp}_{\possp_2}=C^{\possn_2}_{\possp_2}$. So, we get $\possp_2^{\esp-1}\radn_2^\esp=\possn_2$, the latter implying that  $\radn_2=\sqrt[\leftroot{-2}\uproot{2}\esp]{\possn_2\over{\possp_2^{\esp-1}}}$. If $\lcm(\possp_1,\possp_2)=\possp_2$, both $C^{\radn_1}_{\radp_1}$ and $C^{\radn_2}_{\radp_2}$ contribute to  $C^{\possn_2}_{\possp_2}$. In particular, it holds that $(C^{\radn_1}_{\radp_1} \oplus C^{\radn_2}_{\radp_2})^\esp=C^{\possn_1}_{\possp_1} \oplus C^{\possn_2}_{\possp_2}$. By Proposition~\ref{prop:multinomial}, one finds
\[
    (C^{\radn_1}_{\radp_1})^\esp \oplus\; \bigoplus_{i=1}^{\esp-1} \binom{\esp}{i} (C^{\radn_1}_{\radp_1})^i\odot (C^{\radn_2}_{\radp_2})^{\esp-i}
    \; \oplus
    (C^{\radn_2}_{\radp_2})^\esp
    =C^{\possn_1}_{\possp_1} \oplus C^{\possn_2}_{\possp_2}\enspace.
\]
 Since $(C^{\radn_1}_{\radp_1})^\esp=C^{\possn_1}_{\possp_1}$ and by Corollary~\ref{lem:s1} and Proposition~\ref{prop:prod}, that can be rewritten as follows
 \[
    C^{{\radp_2}^{\esp-1}{\radn_2}^{\esp}}_{\radp_2} \oplus \bigoplus_{i=1}^{\esp-1} \binom{\esp}{i} C^{\frac{1}{\lcm(\radp_1,\radp_2)}\cdot {\radp_1}^{i}{\radn_1}^i \cdot {\radp_2}^{\esp-i}{\radn_2}^{\esp-i}}_{\lcm(\radp_1,\radp_2)} =C^{\possn_2}_{\possp_2}
\]
    Recalling that 
    $\lcm(\possp_1,\possp_2) = \possp_2$, $\radp_1=\possp_1$, and $\radp_2=\possp_2$,  the latter equality is true  iff 
    \begin{equation*}
      {\possp_2}^{\esp-1}{\radn_2}^{\esp} +\sum_{i=1}^{\esp-1} \binom{\esp}{i} {\possp_1}^{i}{\radn_1}^i {\possp_2}^{\esp-i-1} \cdot {\radn_2}^{\esp-i}=\possn_2\enspace,
    \end{equation*}
    \ie, once both sides are first multiplied by $\possp_2$ and then added to the term 
    $(\possp_1\radn_1)^{\esp}$, iff
\begin{equation*}
   ( {\possp_1} {\radn_1} + {\possp_2}{\radn_2} )^{\esp} =\possp_2\possn_2 + (\possp_1\radn_1)^{\esp} \enspace.
    \end{equation*}
By $(i)$ and $(iii)$, we get 
\begin{equation*}
  \radn_2  =\frac{\sqrt[\leftroot{-2}\uproot{2}\esp]{\possp_2\possn_2 + \possp_1\possn_1} -  \sqrt[\leftroot{-2}\uproot{2}\esp]{\possp_1\possn_1}}{\possp_2} \enspace.
    \end{equation*}
\end{proof}
The following theorem explains how to compute $\radn_{i+1}$ and  $\radp_{i+1}$ once $\radn_{1}, \ldots, \radn_{i}$ and $\radp_{1}, \ldots, \radp_{i}$ are also known.
\begin{theorem}\label{th:calcolo-n_i-p_i}
Let $\mathring{x}=C^{\radn_1}_{\radp_1} \oplus  \ldots \oplus C^{\radn_l}_{\radp_l}$ be a solution of the equation $\mathring{x}^{\esp}=C^{\possn_1}_{\possp_1} \oplus  \ldots \oplus C^{\possn_h}_{\possp_h}$. 
For any fixed 
natural $i$ with $2\leq i < l$, if $\radn_1, \ldots, \radn_i$,  $\radp_1, \ldots, \radp_i$ are known and  $t\in\set{i,\ldots, h}$, $\possn'_1, \ldots, \possn'_t$, $\possp'_1 , \ldots, \possp'_t$ are positive integers such that $(C^{\radn_1}_{\radp_1} \oplus C^{\radn_2}_{\radp_2} \oplus \ldots \oplus C^{\radn_i}_{\radp_i})^\esp=C^{\possn'_1}_{\possp'_1} \oplus C^{\possn'_2}_{\possp'_2} \oplus \ldots \oplus C^{\possn'_t}_{\possp'_t}$, then the following facts hold:
\begin{enumerate}
\item[(1)]
$
\radp_{i+1}=\possp_{\xi_{i+1}}\enspace,
$ 
\smallskip

where $\xi_{i+1}=\min\left\{j \in \{1, \ldots,h\} \text{ with } \possp_j>\radp_i\, \left|\right. \, \possp_j>\possp'_t \lor  (\possp_j=\possp'_z \text{ for some } 1\leq z\leq t \text{ with } \possn'_z<\possn_j )\right\}$;
\item[(2)]
$
\radn_{i+1} =
\begin{cases}
 \begin{aligned}\frac{\sqrt[w]{{\possp}_{\xi_{i+1}} {\possn}_{\xi_{i+1}}+Q^{*}_i}-\sum_{j=1}^{i} \radp_j \radn_j}{\possp_{\xi_{i+1}}}, \end{aligned} & \text{if }  \lcm(\radp_1, \ldots, \radp_{i+1})=\radp_{i+1}, \\

\begin{aligned}\frac{\sqrt[w]{\possp_{\xi_{i+1}} \possn_{\xi_{i+1}} +Q^{**}_i}-\sum_{e=1}^{j-1} {\radp}_{i_e} {\radn}_{i_e}} {\possp_{\xi_{i+1}}},\end{aligned} & \text{otherwise}\enspace,
\end{cases}
$
\smallskip

where 

\smallskip
$
\begin{aligned}
Q^{*}_i &= \sum_{\substack{k_1+...+k_{i}=\esp \\ 0\leq k_1, \ldots , k_{i} \leq \esp \\ 
    \lambda^*_{i}\neq \radp_{i+1}
    } 
    }^{} \binom{\esp}{k_1, \ldots , k_i} 
    \prod_{t=1}^{i}(\radp_t \radn_t)^{k_t}
    \enspace,
\end{aligned}
$
\smallskip

with $\lambda^*_i$ as in Proposition~\ref{general_newt},
\smallskip


$
\begin{aligned}
Q^{**}_i
   = \sum_{\substack{k_{i_1}+\ldots+k_{i_{j-1}}=\esp \\ 0\leq k_{i_1}, \ldots , k_{i_{j-1}} \leq \esp \\ 
    {\lambda}^{**}_{i_{j-1}}\neq \radp_{i+1}
    } 
    }^{} \binom{\esp}{k_{i_1}, \ldots , k_{i_{j-1}}}  
    \prod_{t=1}^{j-1}(\radp_{i_t} \radn_{i_t})^{k_{i_t}}\enspace, 
\end{aligned}
$
\smallskip

and, regarding $Q^{**}_i$, the set $\{i_1, \ldots, i_j\}$ is the maximal subset of $\{1, \ldots, i+1\}$ such that $i_1<\ldots<i_j$, $i_j=i+1$, and  $\radp_{i_e}$ divides $\radp_{i+1}$ for each  $1\leq e \leq j$ (\ie, $\lcm(\radp_{i_1}, \ldots, \radp_{i_j})=\radp_{i+1}$), and where,   for each  $1\leq e \leq j$ and for any tuple $k_{i_1}, \ldots, k_{i_e}$, $\lambda^{**}_{i_e}$ denotes the $\lcm$ of those $p_{i_{\varepsilon}}$ with $\varepsilon\in\{1,\ldots, e\}$ and $k_{i_{\varepsilon}}\neq 0$ (while $\lambda^{**}_{i_e}=1$ iff all $k_{i_{\varepsilon}}= 0$).
\end{enumerate}
\end{theorem}
\begin{proof}
(1) We deal with the following two mutually exclusive cases \textit{a)} and \textit{b)}. 

\textit{Case a)}: for some $j\in\{1,\ldots, h\}$ the following condition holds: there exists $z\in\{1, \ldots, t\}$ such that $\possp_j=\possp'_z$ and $\possn'_z<\possn_j$. This means that, when the $w$-th power is performed, cycles from the part $(C^{\radn_1}_{\radp_1} \oplus \ldots \oplus C^{\radn_i}_{\radp_i})$ of the solution give rise to a number $\possn'_z$ of cycles of length  $\possp'_z=\possp_j$ where $\possn'_z$ is lower than the number $\possn_j$ of cycles of length $\possp_j$ that are expected once the $w$-th power of the whole solution is computed. Consider the minimum  among all the indexes $j$ satisfying the above introduced condition. It is clear that $\xi_{i+1}$ is just such a minimum and $o_{\xi_{i+1}}$ is the minimum among the values $o_j$ corresponding to those indexes $j$. Since by Corollary~\ref{lem:s1} and regarding each $j$ satisfying the above mentioned condition the $w$-th power of cycles of length $\possp'_z=\possp_j$ gives rise to cycles of length $\possp'_z$, by item $(i)$ of  Proposition~\ref{prop:base} $\radp_{i+1}$ comes from the set $\{o_1, \ldots, o_h\}$, 
and it is the successor of $\radp_{i}$, we get that $p_{i+1}$ can be nothing but $o_{\xi_{i+1}}$, 
or, equivalently, $i+1=\xi_{i+1}$. Indeed, according to Proposition~\ref{general_newt}, if cycles of length greater than $o_{\xi_{i+1}}$ were added to the part $(C^{\radn_1}_{\radp_1} \oplus \ldots \oplus C^{\radn_i}_{\radp_i})$ of the solution instead of cycles of length $\possp_{\xi_{i+1}}$, they would give rise to cycles of greater length, barring the generation of the missing cycles of length $\possp_{\xi_{i+1}}$.

\textit{Case b)}: there is no index $j\in\{1,\ldots, h\}$ satisfying the above mentioned condition.  Similar arguments from case \textit{a)} over the values $\possp_j$ and the corresponding indexes $j$ such that $\possp_j>\possp'_t$ lead to the conclusion that $\xi_{i+1}$ is the minimum of such indexes, $i+1=\xi_{i+1}$, and $\radp_{i+1}=\possp_{\xi_{i+1}}$.
\smallskip\\
(2) We deal with the following two mutually exclusive cases:

\textit{Case 2.1)}: $\lcm(\radp_1, \ldots, \radp_{i+1})=\radp_{i+1}$.     
By Proposition~\ref{general_newt}, we can write 
\begin{align*}
(C^{\radn_1}_{\radp_1} \oplus \ldots \oplus C^{\radn_{i+1}}_{\radp_{i+1}})^\esp
&=
\bigoplus\limits_{\substack{k_1+...+k_{i+1}=\esp\\0\leq k_1,\ldots,k_{i+1}  \leq \esp }}^{}\binom{\esp}{k_1,k_2,...,k_{i+1}}C^{\frac{1}{\lambda^*_{i+1}} \prod_{t=1}^{i+1} (\radp_t \radn_t)^{k_t}}_{\lambda^*_{i+1}}
\end{align*}

Among all the addends of the latter sum, only the ones with a multinomial coefficient defined by $k_1, \ldots, k_{i+1}$ such that $\lambda^*_{i+1}=\radp_{i+1}$ give rise to cycles of length $\radp_{i+1}$, where $\radp_{i+1}=\possp_{\xi_{i+1}}$. In particular, it holds that
\[
\bigoplus\limits_{\substack{k_1+...+k_{i+1}=\esp\\0\leq k_1,\ldots,k_{i+1}  \leq \esp\\ \lambda^*_{i+1}=\radp_{i+1} }}^{}\binom{\esp}{k_1,\ldots,k_{i+1}}C^{\frac{1}{\lambda^*_{i+1}} \prod_{t=1}^{i+1} (\radp_t \radn_t)^{k_t}}_{\lambda^*_{i+1}}= C^{\possn_{\xi_{i+1}}}_{\possp_{\xi_{i+1}}}\enspace, 
\]
and, hence, 
\begin{equation}
\label{pezzoprincipale}
\sum_{\substack{k_1+\ldots+k_{i+1}=\esp \\ 0\leq k_1,k_2, \ldots , k_{i+1} \leq \esp \\ 
    \lambda^*_{i+1}=\radp_{i+1}
    } 
    }^{} \binom{\esp}{k_1, \ldots , k_{i+1}} \cdot \frac{1}{\lambda^*_{i+1}} \cdot \prod_{t=1}^{i+1} (\radp_t \radn_t)^{k_t}= \possn_{\xi_{i+1}}\enspace.
\end{equation}
Since $\lambda^*_{i+1}=\possp_{\xi_{i+1}}$, 
when both sides of~Equation~\eqref{pezzoprincipale} are first multiplied by $ \possp_{\xi_{i+1}}$ and then summed to the quantity
\begin{equation*}
\begin{aligned}
\sum_{\substack{k_1+\ldots +k_{i+1}=\esp \\ 0\leq k_1,\ldots , k_{i+1} \leq \esp \\ 
    \lambda^*_{i+1}\neq \radp_{i+1}\, \wedge \, k_{i+1}=0
    } 
    }^{} \binom{\esp}{k_1,\ldots , k_{i+1}}   \prod_{t=1}^{i+1} (\radp_t \radn_t)^{k_t}=
\sum_{\substack{k_1+\ldots +k_{i}=\esp \\ 0\leq k_1,\ldots , k_{i} \leq \esp \\ 
    \lambda^*_{i}\neq \radp_{i+1}
    } 
    }^{} \binom{\esp}{k_1,\ldots , k_i}  \prod_{t=1}^{i} (\radp_t \radn_t)^{k_t} = Q^{*}_i
   \enspace,
\end{aligned}
\end{equation*}
Equation~\eqref{pezzoprincipale} becomes 
\begin{equation}
\label{sommatotale}
\sum_{\substack{k_1+\ldots+k_{i+1}=\esp \\ 0\leq k_1, \ldots , k_{i+1} \leq \esp    
    } 
    }^{} \binom{\esp}{k_1, \ldots , k_{i+1}} \prod_{\substack{t=1}}^{i+1} \radp_t^{k_t} \radn_t^{k_t}  = \possp_{\xi_{i+1}} \possn_{\xi_{i+1}} + Q^{*}_i\enspace.
 \end{equation}
Indeed, by the assumption that $\lcm(\radp_1, \ldots, \radp_{i+1})=\radp_{i+1}$, there can be  no tuple $(k_1, \ldots , k_{i+1})$ from the sum of Equation~\eqref{sommatotale} such that both the conditions $k_{i+1}\neq 0$ and $\lambda^*_{i+1}\neq \radp_{i+1}$ hold.
Now, Equation~\eqref{sommatotale} can be rewritten as 
\[
(\radp_1 \radn_1 +\ldots \radp_i \radn_i + \radp_{i+1} \radn_{i+1})^{\esp}= \possp_{\xi_{i+1}} \possn_{\xi_{i+1}}+Q^{*}_i\enspace.
\]
Since $Q^{*}_i$ does not depend on $\radn_{i+1}$, and, in particular, $Q^{*}_i$ can be computed on the basis of  $\radn_{1}, \ldots, \radn_{i}$,  we get 
\[
\radn_{i+1} =\frac{\sqrt[w]{\possp_{\xi_{i+1}} \possn_{\xi_{i+1}}+Q^{*}_i}-\sum_{j=1}^{i} \radp_j \radn_j}{\possp_{\xi_{i+1}}}
\]

\textit{Case 2.2)}: $\lcm(\radp_1, \ldots, \radp_{i+1})>\radp_{i+1}$.  
When computing the $w$-th power of $\mathring{x}$ the set $C^{\possn_\xi}_{\possp_\xi} = C^{\possn_{\xi}}_{\radp_{i+1}}$ can be formed only by the contribution of those sets $C^{\radn_{i_1}}_{\radp_{i_1}}$, \ldots, $C^{\radn_{i_j}}_{\radp_{i_j}}$ (including $C^{\radn_{i+1}}_{\radp_{i+1}}$)  such that $i_1<\ldots<i_j$, $i_j=i+1$, and $\radp_{i_e}$ divides $\radp_{i+1}$ for each  $1\leq e \leq j$. Since $\lcm(\radp_{i_1}, \ldots, \radp_{i_j})=\radp_{i+1}$, we can proceeding in the same way as the case \textit{2.1)} but with the indexes $i_1$, \ldots, $i_j$ instead of $1$, \ldots, $i+1$, respectively. 
Therefore, it holds that
\begin{equation}
\label{sommatotale2}
\sum_{\substack{k_{i_1}+\ldots+k_{i_j}=\esp \\ 0\leq k_{i_1}, \ldots , k_{i_j} \leq \esp    
    } 
    }^{} \binom{\esp}{k_{i_1}, \ldots , k_{i_j}} \prod_{\substack{t=1}}^{j} {\radp}_{i_t}^{k_{i_t}} {\radn}_{i_t}^{k_{i_t}}  = \possp_{\xi_{i+1}} \possn_{\xi_{i+1}} + Q^{**}_i\enspace,
 \end{equation}
 where 
 \begin{equation*}
\begin{aligned}
Q^{**}_i
   = \sum_{\substack{k_{i_1}+\ldots+k_{i_{j-1}}=\esp \\ 0\leq k_{i_1}, \ldots , k_{i_{j-1}} \leq \esp \\ 
    {\lambda}^{**}_{i_{j-1}}\neq \radp_{i+1}
    } 
    }^{} \binom{\esp}{k_{i_1}, \ldots , k_{i_{j-1}}} 
    \prod_{t=1}^{j-1}(\radp_{i_t} \radn_{i_t})^{k_{i_t}}\enspace, 
\end{aligned}
\end{equation*}
and, hence, Equation~\eqref{sommatotale2} can be rewritten as 
\[
(\radp_{i_1} \radn_{i_1} +\ldots {\radp}_{i_{j-1}} {\radn}_{i_{j-1}} + {\radp}_{i_j} {\radn}_{i_j})^{\esp}= \possp_{\xi_{i+1}} \possn_{\xi_{i+1}}+Q^{**}_i\enspace.
\]
Since $Q^{**}_i$ does not depend on $\radn_{i+1}={\radn}_{i_j}$, and, in particular, $Q^{**}_i$ can be computed on the basis of  ${\radn}_{i_1}, \ldots, {\radn}_{i_{j-1}}$,  we get 
\[
\radn_{i+1} =
\begin{aligned}\frac{\sqrt[w]{\possp_{\xi_{i+1}} \possn_{\xi_{i+1}} +Q^{**}_i}-\sum_{e=1}^{j-1} {\radp}_{i_e} {\radn}_{i_e}} {\possp_{\xi_{i+1}}},\end{aligned}
\]



\end{proof}
At this point it is clear that the $w$-th root of a DDS
is always unique, if it exists (\ie, if all $n_i$'s turn out to be natural numbers). 

\section{Intersection between abstractions}\label{sec:intersection}
\definecolor{ared}{RGB}{170, 7, 7}
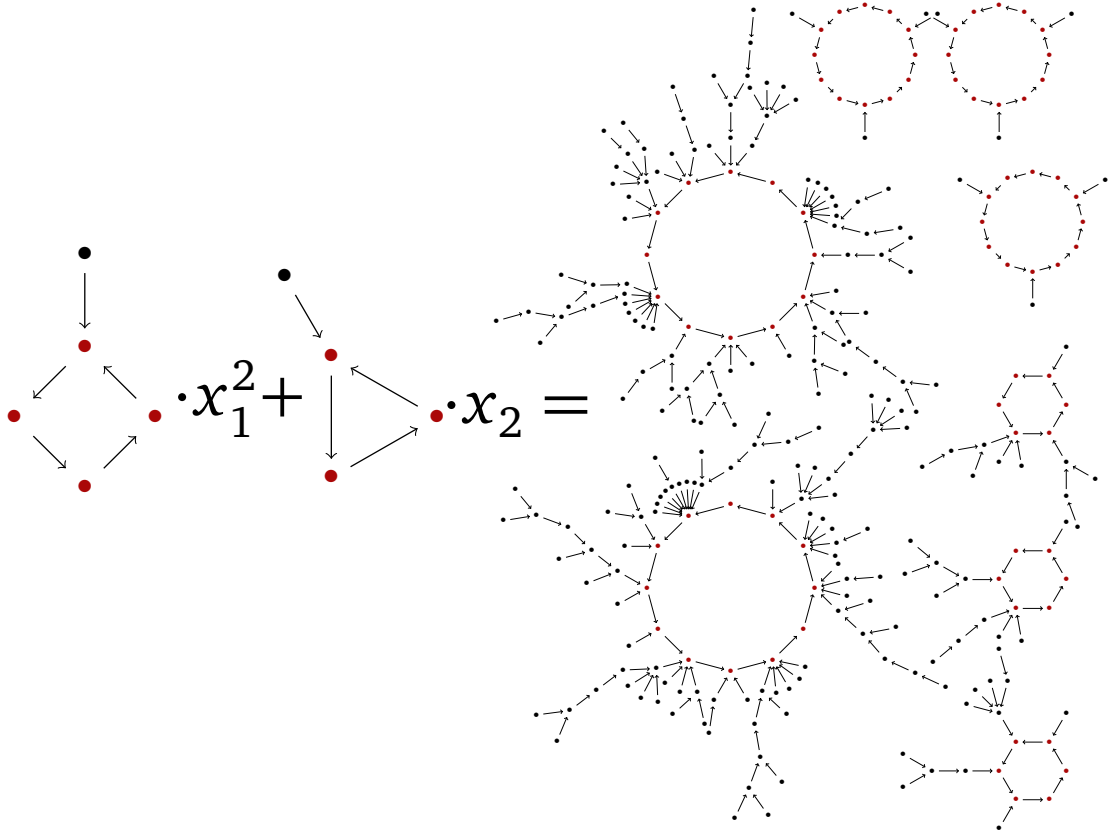
\begin{figure}[htb]
\begin{tikzpicture}
\node[] at (-10,0){\space
\resizebox{\textwidth/2}{!}{%
\begin{tikzpicture}
    \node[draw=none,fill=none, scale=2, anchor=south] (0) at (-2.35,-.455) {$\cdot x_1^2+$};
    \node[draw=none,fill=none, scale=2, anchor=south] (1) at (.6,-.455) {$\cdot x_2=$}; 
    \node[text=white](a) at (0,5.5){$\bullet$};
    \node[text=white](a) at (0,-5.5){$\bullet$}; 
    \begin{oodgraph}
        \addcycle[xshift=-4cm,nodes prefix = a, color=ared,radius=0.75cm]{4};
        \addbeard[attach node = a->2]{1};
        
        \addcycle[nodes prefix = b, color=ared,radius=0.75cm,xshift=-1cm]{3};
        \addbeard[attach node = b->2]{1};
    \end{oodgraph}
\end{tikzpicture}
}};
\node[] at (-3.5,0){\space
\resizebox{\textwidth/2}{!}{%
\begin{tikzpicture}
	\begin{oodgraph}
	    \addcycle[xshift=8.00cm, yshift=6cm, nodes prefix = a, color=ared, radius=1.5cm]{12};
		\addbeard[attach node = a->2]{1};
		\addbeard[attach node = a->6]{1};
		\addbeard[attach node = a->10]{1};

		\addcycle[xshift=4cm, yshift=6cm, nodes prefix = b, color=ared, radius=1.5cm]{12};
		\addbeard[attach node = b->2]{1};
		\addbeard[attach node = b->6]{1};
		\addbeard[attach node = b->10]{1};

		\addcycle[xshift=9cm, yshift=1cm, nodes prefix = c, color=ared, radius=1.5cm]{12};
		\addbeard[attach node = c->2]{1};
		\addbeard[attach node = c->6]{1};
		\addbeard[attach node = c->10]{1};

		\addcycle[xshift=9cm, yshift=-4.5cm, nodes prefix = d, color=ared]{6};
		\addbeard[attach node = d->2]{1};
		\addbeard[attach node = d->5]{3};
		\addbeard[attach node = d->6]{1};
		\addbeard[attach node = d->5->1]{2};
		\addbeard[attach node = d->6->1]{2};
		\addbeard[attach node = d->5->1->1]{1};
		\addbeard[attach node = d->6->1->1]{1};

		\addcycle[xshift=9cm, yshift=-9.75cm, nodes prefix = e, color=ared]{6};
		\addbeard[attach node = e->2]{1};
		\addbeard[attach node = e->4]{1};
		\addbeard[attach node = e->5]{3};
		\addbeard[attach node = e->4->1]{2};
		\addbeard[attach node = e->5->1]{1};
		\addbeard[attach node = e->4->1->1]{2};
		\addbeard[attach node = e->5->1->1]{1};

		\addcycle[xshift=9cm, yshift=-15.5cm, nodes prefix = f, color=ared]{6};
		\addbeard[attach node = f->2]{1};
		\addbeard[attach node = f->3]{1};
		\addbeard[attach node = f->4]{1};
		\addbeard[attach node = f->5]{1};
		\addbeard[attach node = f->3->1]{4};
		\addbeard[attach node = f->4->1]{1};
		\addbeard[attach node = f->3->1->1]{1};
		\addbeard[attach node = f->4->1->1]{2};

	    \addcycle[nodes prefix = g, radius=2.5cm, color=ared]{12};
		\addbeard[attach node = g->1]{1};
		\addbeard[attach node = g->2]{7};
		\addbeard[attach node = g->4]{3};
		\addbeard[attach node = g->5]{3};
		\addbeard[attach node = g->6]{3};
		\addbeard[attach node = g->8]{7};
		\addbeard[attach node = g->9]{1};
		\addbeard[attach node = g->10]{3};
		\addbeard[attach node = g->11]{1};
		\addbeard[attach node = g->12]{3};
		\addbeard[attach node = g->1->1]{1};
		\addbeard[attach node = g->2->1]{2};
		\addbeard[attach node = g->4->1,rotation angle=-15]{1};
		\addbeard[attach node = g->4->2]{1};
		\addbeard[attach node = g->5->1]{1};
		\addbeard[attach node = g->6->1]{4};
		\addbeard[attach node = g->8->1, rotation angle=-15]{1};
		\addbeard[attach node = g->8->2]{1};
		\addbeard[attach node = g->9->1]{2};
		\addbeard[attach node = g->10->1]{2};
		\addbeard[attach node = g->12->1,rotation angle=-20]{2};
		\addbeard[attach node = g->12->2]{2};
		\addbeard[attach node = g->1->1->1]{2};
		\addbeard[attach node = g->2->1->1]{1};
		\addbeard[attach node = g->2->1->2]{1};
		\addbeard[attach node = g->4->1->1]{4};
		\addbeard[attach node = g->4->2->1]{2};
		\addbeard[attach node = g->5->1->1]{1};
		\addbeard[attach node = g->6->1->1]{1};
		\addbeard[attach node = g->6->1->2]{1};
		\addbeard[attach node = g->8->1->1]{2};
		\addbeard[attach node = g->8->2->1]{1};
		\addbeard[attach node = g->9->1->1]{1};
		\addbeard[attach node = g->10->1->1]{2};
		\addbeard[attach node = g->10->1->2]{2};
		\addbeard[attach node = g->12->1->1]{2};
		\addbeard[attach node = g->12->2->1]{1};
		\addbeard[attach node = g->4->2->1->1]{1};
		\addbeard[attach node = g->8->2->1->1]{2};
		\addbeard[attach node = g->12->2->1->1]{1};
		\addbeard[attach node = g->4->2->1->1->1]{1};
		\addbeard[attach node = g->8->2->1->1->1]{1};
		\addbeard[attach node = g->12->2->1->1->1]{2};
		
		\addcycle[yshift=-10.00cm, nodes prefix = h, radius=2.5cm, color=ared]{12};
		\addbeard[attach node = h->1]{4};
		\addbeard[attach node = h->2]{5};
		\addbeard[attach node = h->3]{2};
		\addbeard[attach node = h->5]{9};
		\addbeard[attach node = h->6]{2};
		\addbeard[attach node = h->7]{2};
		\addbeard[attach node = h->8]{1};
		\addbeard[attach node = h->9]{4};
		\addbeard[attach node = h->10]{2};
		\addbeard[attach node = h->11]{5};
		\addbeard[attach node = h->1->1]{2};
		\addbeard[attach node = h->1->3]{1};
		\addbeard[attach node = h->2->2]{1};
		\addbeard[attach node = h->3->1]{4};
		\addbeard[attach node = h->5->1,rotation angle=-45]{2};
		\addbeard[attach node = h->5->5]{1};
		\addbeard[attach node = h->6->1]{2};
		\addbeard[attach node = h->7->1]{2};
		\addbeard[attach node = h->9->1]{4};
		\addbeard[attach node = h->9->4]{2};
		\addbeard[attach node = h->10->1]{1};
		\addbeard[attach node = h->11->1]{2};
		\addbeard[attach node = h->1->1->1]{2};
		\addbeard[attach node = h->3->1->2]{1};
		\addbeard[attach node = h->5->1->1,rotation angle=-45]{1};
		\addbeard[attach node = h->7->1->1]{2};
		\addbeard[attach node = h->9->1->1]{1};
		\addbeard[attach node = h->11->1->1]{1};
		\addbeard[attach node = h->1->1->1->1]{1};
		\addbeard[attach node = h->3->1->2->1]{1};
		\addbeard[attach node = h->5->1->1->1,rotation angle=-45]{2};
		\addbeard[attach node = h->7->1->1->1]{1};
		\addbeard[attach node = h->9->1->1->1]{1};
		\addbeard[attach node = h->11->1->1->1]{2};
		\addbeard[attach node = h->1->1->1->1->1]{1};
		\addbeard[attach node = h->3->1->2->1->1]{4};
		\addbeard[attach node = h->5->1->1->1->1,rotation angle=-45]{1};
		\addbeard[attach node = h->7->1->1->1->1]{2};
		\addbeard[attach node = h->9->1->1->1->1]{2};
		\addbeard[attach node = h->11->1->1->1->1]{2};
	\end{oodgraph}
\end{tikzpicture}
}}; 
\end{tikzpicture}
\vspace{-1cm}
\captionof{figure}{An example of Equation~\ref{eq:problemSolvedOriginal}. The coefficients $a_1$, $a_2$   and the know term $b$ are depicted by their dynamics graphs.}
\label{fig:my_eq}
\end{figure}
Once considered both the c-abstraction and a-abstraction of  Equation~\eqref{eq:problemSolvedOriginal} and provided the two corresponding solution sets, the final step to perform - that we name intersection between abstractions - is combining each solution from the first set with each solution from the second one to establish what resulting pairs lead to a possible solution of  Equation~\eqref{eq:problemSolvedOriginal}. In other words, 
$(x_1, \ldots, x_{\qtvar})$ is a solution candidate  of  Equation~\eqref{eq:problemSolvedOriginal} if each of the tuples  $(|x_1| 
,\ldots , |x_{\qtvar}|)$ and $(\mathring{x_1}, \ldots, \mathring{x}_{\qtvar})$ belongs  to the solution set of the c-abstraction and a-abstraction equation, respectively. Moreover, a solution of the c-abstraction equation can be combined with one of the a-abstraction equation, if for every $i$ the total number of periodic points of $\mathring{x_i}$ is at most $|x_i|$.  Let us illustrate such a final step by the following example.
\begin{example}
\label{ex:final}
Consider the equation 
\[
a_1\cdot x_1^2 + a_2\cdot x_2=b 
\]
where $a_1$, $a_2$, and $b$ are as in Figure~\ref{fig:my_eq}. The corresponding c-abstraction and a-abstraction equations are
\[
5\cdot |x_1|^2 + 4 \cdot |x_2|= 293 \enspace, 
\]
and
\[C^1_4 \odot \mathring{x_1}^2 \oplus C^1_3 \odot \mathring{x_2} = C^3_6 \oplus C^5_{12}\enspace,
\]
respectively. At this point, we aim at enumerating the solutions of both the abstraction equations.
Regarding the c-abstraction one, the MDD of Figure~\ref{fig:mddnodesS5} provides the following solutions:
 \[|x_1|=7,|x_2|=12\]
 \[|x_1|=5,|x_2|=42\]
 \[|x_1|=3,|x_2|=62\]
 \[|x_1|=1,|x_2|=72\]
 
 \begin{figure}[htb]
  \begin{center}
	\begin{tikzpicture}[scale=.3,-latex,auto, semithick , state/.style ={circle,draw,minimum width=8mm},font=\footnotesize]
	\node[state,scale=.65] (r) at (0,0) {$r$};
	\node[state,scale=.65] (1=245) at (-4.5,-4) {$245$};
	\node[state,scale=.65] (1=125) at (-1.5,-4) {$125$};
	\node[state,scale=.65] (1=45) at (1.5,-4) {$45$};
	\node[state,scale=.65] (1=5) at (4.5,-4) {$5$};
    \node[state,scale=.65] (tt) at (0,-8) {$tt$};
    
    \draw [->]  (r) to node[left,scale=.65] {$7$} (1=245);
    \draw [->]  (r) to node[left,scale=.65] {$5$} (1=125);
    \draw [->]  (r) to node[right,scale=.65] {$3$} (1=45);
    \draw [->]  (r) to node[right,scale=.65] {$1$} (1=5);
    \draw [->]  (1=245) to node[left,scale=.65] {$12$} (tt);
    \draw [->]  (1=125) to node[left,scale=.65] {$42$} (tt);
    \draw [->]  (1=45) to node[right,scale=.65] {$62$} (tt);
    \draw [->]  (1=5) to node[right,scale=.65] {$72$} (tt);
    
    \end{tikzpicture}
  \end{center}
	\caption{The reduced MDD representing all the solutions of $5 \cdot |x_1|^2 + 4 \cdot |x_2| =293$. There are $\qtvar=2$ variables. The first level and the corresponding outgoing edges represent the variable $|x_1|$ and its possible values. The second level and the outgoing edges represent $|x_2|$.}
	\label{fig:mddnodesS5}
\end{figure}
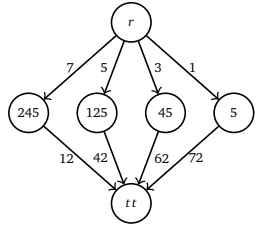

As far the a-abstraction equation is concerned, there are 16 basic equations and, according to the necessary ones, the MDD $CS$ of  Figure~\ref{fig:csS5} provide all the feasible way by which the two monomials  $C_4^1 \odot X_1$  and $C^1_3 \odot X_2$ can concur to form $C^3_6 \oplus C^5_{12}$, where  $X_1 = \mathring{x_1}^2$ and $X_2=\mathring{x}_2$.
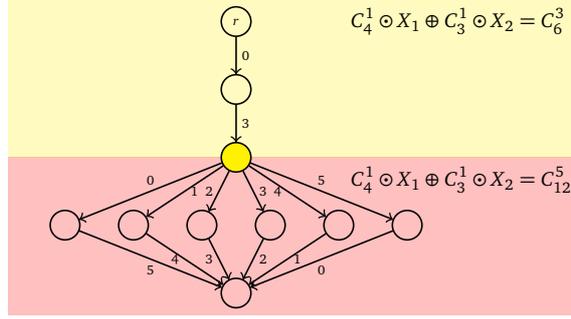
\begin{figure}[t]
  \begin{center}
	\begin{tikzpicture}[scale=.3,-latex,auto, semithick , state/.style ={circle,draw,minimum width=6mm},font=\footnotesize]
	\path [fill=yellow, nearly transparent] (-10,1) rectangle (15,-6);
    \path [fill=red, nearly transparent] (-10,-6) rectangle (15,-13);
	\node[state,scale=.65] (r) at (0,0) {$r$};
	\node[text width=3cm] at (10,0) 
    {$C_4^1\odot X_1 \oplus  C_3^1 \odot X_2=C_6^3$};
    \node[text width=3cm] at (10,-7) 
    {$C_4^1\odot X_1 \oplus  C_3^1 \odot X_2=C_{12}^5$};
	\node[state,scale=.65] (1) at (0,-3) {};
	\node[state,scale=.65,fill=yellow] (tt1) at (0,-6) {};
    \node[state,scale=.65] (2=2) at (-1.5,-9) {};
    \node[state,scale=.65] (2=3) at (1.5,-9) {};
    \node[state,scale=.65] (2=1) at (-4.5,-9) {};
    \node[state,scale=.65] (2=0) at (-7.5,-9) {};
    \node[state,scale=.65] (2=4) at (4.5,-9) {};
    \node[state,scale=.65] (2=5) at (7.5,-9) {};
    \node[state,scale=.65] (tt2) at (0,-12) {};

    \draw [->]  (r) to node[right,scale=.65] {$0$} (1);
    \draw [->]  (1) to node[right,scale=.65] {$3$} (tt1);
    \draw [->]  (tt1) to node[above,scale=.65] {$0$} (2=0);
    \draw [->]  (tt1) to node[right,scale=.65] {$1$} (2=1);
    \draw [->]  (tt1) to node[left,scale=.65] {$2$} (2=2);
    \draw [->]  (tt1) to node[right,scale=.65] {$3$} (2=3);
    \draw [->]  (tt1) to node[left,scale=.65] {$4$} (2=4);
    \draw [->]  (tt1) to node[above,scale=.65] {$5$} (2=5);
    \draw [->]  (2=0) to node[below,scale=.65] {$5$} (tt2);
    \draw [->]  (2=1) to node[left,scale=.65] {$4$} (tt2);
    \draw [->]  (2=2) to node[left,scale=.65] {$3$} (tt2);
    \draw [->]  (2=3) to node[right,scale=.65] {$2$} (tt2);
    \draw [->]  (2=4) to node[right,scale=.65] {$1$} (tt2);
    \draw [->]  (2=5) to node[below,scale=.65] {$0$} (tt2);

	\end{tikzpicture}
  \end{center}
	\caption{
	The MDD $CS=CS_1\times CS_2$ represents all the feasible ways by which, according to the set of necessary equations, the monomials of the a-abstraction equation from Example~\ref{ex:final} can concur to form its right-hand side. According to the cartesian product of MDD, the yellow node is at the same time the \ttnode node of $CS_1$ and the \rootnode node of  $CS_{2}$. In each of the two MDD, the first (resp., second) level corresponds to the monomial $C_4^1\odot X_1$ (resp., $C_3^1 \odot X_2$). The values $val(\alpha)$ associated to nodes are omitted for simplicity.}
	\label{fig:csS5}
\end{figure}
Namely, only the monomial $C^1_3 \odot \mathring{x}_2$ contributes to form $C^3_6$ (see $CS_1$), while there are several ways by which both of them contribute to form $C^5_{12}$. If among the necessary equations involving $X_1=\mathring{x}_1^2$, one considers only those admitting a non empty set of solutions $x_1$, \ie, after the computation of the square root of the values of $X_1$ has been performed too, the following two feasible Systems~\eqref{systelcontr} remain:
\[
\begin{cases}
C_3^1 \odot \mathring{x}_2=C^3_6 \oplus C^5_{12}
\end{cases}
\]
\[
\begin{cases}
C_4^1 \odot \mathring{x_1}^2=C^3_{12}\\
C_3^1 \odot \mathring{x}_2=C^3_6 \oplus C^2_{12}
\end{cases}
\]
In both cases, the values of $X_2=\mathring{x}_2$ are computed by a Cartesian products of SB-MDDs.
Due to the form of $\mathring{a}_1$ and $\mathring{a}_2$ and the fact that the two monomials contain distinct variables, in each system there are no equations involving the same variable. Hence, no intersection operation between solutions of   equations is required. The solutions of the  a-abstraction equation are: 
\begin{align*}
    \mathring{x_1}=C^1_3 & , \mathring{x_2}=C^1_6 \oplus C^2_4\\
\mathring{x_1}=C^1_3 & , \mathring{x_2}=C^3_2 \oplus C^2_4\\
\mathring{x_1}=\text{\MVZero} & , \mathring{x_2}=C^1_6 \oplus C^1_{12} \oplus C^2_4\\
\mathring{x_1}=\text{\MVZero} & , \mathring{x_2}=C^1_6 \oplus C^5_4\\
\mathring{x_1}=\text{\MVZero} & , \mathring{x_2}=C^3_2 \oplus C^1_{12} \oplus C^2_4\\
\mathring{x_1}=\text{\MVZero} & , \mathring{x_2}=C^3_2 \oplus C^5_4
\end{align*}
Some solutions $(\mathring{x_1}, \mathring{x_2})$ of the a-abstraction equation can be coupled to no solution $(|x_1|,|x_2|)$ of the c-abstraction equation to lead a solution of the given original equation. Namely, by the solutions of the c-abstraction equation, $x_1$ necessarily has at least one state. Therefore, the only possible value of $\mathring{x_1}$ is $C^1_3$. This implies that $x_1$ must have at least 3 states and $|x_2| \geq 14$ (since $\mathring{x_2}$ consists of $14$ periodic points). Then, the solutions $(|x_1|=1,|x_2|=72)$ and $(|x_1|=7,|x_2|=12)$ of the c-abstraction equation can not be coupled with any solution of the c-abstraction equation. 
This process leads to the identification of the following candidate solutions of the given original equation:
\[(x_1,x_2) \in \ring^2 \text{ s.t. }(\mathring{x_1}=C^1_3) \text{ and } (\mathring{x_2} \in \set{C^1_6 \oplus C^2_4,C^3_2 \oplus C^2_4}) \text{ and } ((|x_1|=3 \land |x_2|=62) \text{ or } (|x_1|=5 \land |x_2|=72))\]

\end{example}

\section{Conclusion}
This paper presents a complete algorithmic pipeline
for solving both the $c$- and $a$-abstractions of polynomial
equations (with constant right-hand term) over DDS.
The pipeline includes a number of subtleties  
allowing reasonable performances that are compatible with
practical applications.

Devising an algorithm that solves  in an efficient way the $t$-abstraction of an equation over DDS is certainly the  main step for further researches concerning this subject. Actually, 
this is a rather complex task. 

A further interesting research direction consists in
trying to understand the precise computational
complexity of problems that arise when considering the different tasks of the pipeline. For example, what is the computational
complexity of establishing whether  a basic equation has solutions? It is clear
that the problem is in \NP\ but we conjecture that in
fact it is in \P. Along the same line of thoughts,
one finds that the problem of enumerating the solutions of a basic equation is in \EnumP\, but
is it complete for this class? Now, stepping to
the more complex problem of deciding whether
an $a$-abstraction equation admits a solution, what is precisely its complexity class?

\bibliographystyle{elsarticle-harv}
\bibliography{biblio.bib}

\begin{thebibliography}{17}
\expandafter\ifx\csname natexlab\endcsname\relax\def\natexlab#1{#1}\fi
\providecommand{\url}[1]{\texttt{#1}}
\providecommand{\href}[2]{#2}
\providecommand{\path}[1]{#1}
\providecommand{\DOIprefix}{doi:}
\providecommand{\ArXivprefix}{arXiv:}
\providecommand{\URLprefix}{URL: }
\providecommand{\Pubmedprefix}{pmid:}
\providecommand{\doi}[1]{\href{http://dx.doi.org/#1}{\path{#1}}}
\providecommand{\Pubmed}[1]{\href{pmid:#1}{\path{#1}}}
\providecommand{\bibinfo}[2]{#2}
\ifx\xfnm\relax \def\xfnm[#1]{\unskip,\space#1}\fi
\bibitem[{Adamatzky et~al.(2020)Adamatzky, Goles, Mart{\'{\i}}nez, Tsompanas,
  Tegelaar and Wosten}]{Adamatzky0MTTW20}
\bibinfo{author}{Adamatzky, A.}, \bibinfo{author}{Goles, E.},
  \bibinfo{author}{Mart{\'{\i}}nez, G.J.}, \bibinfo{author}{Tsompanas, M.I.},
  \bibinfo{author}{Tegelaar, M.}, \bibinfo{author}{Wosten, H.A.B.},
  \bibinfo{year}{2020}.
\newblock \bibinfo{title}{Fungal automata}.
\newblock \bibinfo{journal}{Complex Syst.} \bibinfo{volume}{29}.
\newblock \URLprefix
  \url{https://www.complex-systems.com/abstracts/v29\_i04\_a02/}.
\bibitem[{Alonso{-}Sanz(2012)}]{AlonsoSanz12}
\bibinfo{author}{Alonso{-}Sanz, R.}, \bibinfo{year}{2012}.
\newblock \bibinfo{title}{Cellular automata and other discrete dynamical
  systems with memory}, in: \bibinfo{editor}{Smari, W.W.},
  \bibinfo{editor}{Zeljkovic, V.} (Eds.), \bibinfo{booktitle}{Proceedings of
  {HPCS}}, \bibinfo{publisher}{{IEEE}}. p. \bibinfo{pages}{215}.
\bibitem[{Aracena et~al.(2021)Aracena, Cabrera{-}Crot and
  Salinas}]{AracenaCS21}
\bibinfo{author}{Aracena, J.}, \bibinfo{author}{Cabrera{-}Crot, L.},
  \bibinfo{author}{Salinas, L.}, \bibinfo{year}{2021}.
\newblock \bibinfo{title}{Finding the fixed points of a boolean network from a
  positive feedback vertex set}.
\newblock \bibinfo{journal}{Bioinform.} \bibinfo{volume}{37},
  \bibinfo{pages}{1148--1155}.
\newblock \URLprefix \url{https://doi.org/10.1093/bioinformatics/btaa922},
  \DOIprefix\doi{10.1093/bioinformatics/btaa922}.
\bibitem[{Bergman et~al.(2014)Bergman, Cire and van Hoeve}]{bergman2014mdd}
\bibinfo{author}{Bergman, D.}, \bibinfo{author}{Cire, A.A.},
  \bibinfo{author}{van Hoeve, W.}, \bibinfo{year}{2014}.
\newblock \bibinfo{title}{{MDD} propagation for sequence constraints}.
\newblock \bibinfo{journal}{Journal of Artificial Intelligence Research}
  \bibinfo{volume}{50}, \bibinfo{pages}{697--722}.
\bibitem[{Bergman et~al.(2016)Bergman, Cire, Van~Hoeve and
  Hooker}]{bergman2016decision}
\bibinfo{author}{Bergman, D.}, \bibinfo{author}{Cire, A.A.},
  \bibinfo{author}{Van~Hoeve, W.J.}, \bibinfo{author}{Hooker, J.},
  \bibinfo{year}{2016}.
\newblock \bibinfo{title}{Decision diagrams for optimization}.
  volume~\bibinfo{volume}{1}.
\newblock \bibinfo{publisher}{Springer}.
\bibitem[{Bower and Bolouri(2004)}]{bower2004computational}
\bibinfo{author}{Bower, J.M.}, \bibinfo{author}{Bolouri, H.},
  \bibinfo{year}{2004}.
\newblock \bibinfo{title}{Computational modeling of genetic and biochemical
  networks}.
\newblock \bibinfo{publisher}{MIT press}.
\bibitem[{Chaudhuri et~al.(1997)Chaudhuri, Chowdhury, Nandi and
  Chattopadhyay}]{CCNC97}
\bibinfo{author}{Chaudhuri, P.}, \bibinfo{author}{Chowdhury, D.},
  \bibinfo{author}{Nandi, S.}, \bibinfo{author}{Chattopadhyay, S.},
  \bibinfo{year}{1997}.
\newblock \bibinfo{title}{Additive Cellular Automata Theory and Applications}.
  volume~\bibinfo{volume}{1}.
\newblock \bibinfo{publisher}{IEEE Press}.
\bibitem[{Darwiche and Marquis(2002)}]{darwiche2002knowledge}
\bibinfo{author}{Darwiche, A.}, \bibinfo{author}{Marquis, P.},
  \bibinfo{year}{2002}.
\newblock \bibinfo{title}{A knowledge compilation map}.
\newblock \bibinfo{journal}{Journal of Artificial Intelligence Research}
  \bibinfo{volume}{17}, \bibinfo{pages}{229--264}.
\bibitem[{Demongeot et~al.(2022)Demongeot, Melliti, Noual, Regnault and
  Sen{\'{e}}}]{DemongeotMNRS22}
\bibinfo{author}{Demongeot, J.}, \bibinfo{author}{Melliti, T.},
  \bibinfo{author}{Noual, M.}, \bibinfo{author}{Regnault, D.},
  \bibinfo{author}{Sen{\'{e}}, S.}, \bibinfo{year}{2022}.
\newblock \bibinfo{title}{On boolean automata isolated cycles and tangential
  double-cycles dynamics}, in: \bibinfo{editor}{Adamatzky, A.} (Ed.),
  \bibinfo{booktitle}{Automata and Complexity - Essays Presented to Eric Goles
  on the Occasion of His 70th Birthday}, \bibinfo{publisher}{Springer}. pp.
  \bibinfo{pages}{145--178}.
\newblock \URLprefix \url{https://doi.org/10.1007/978-3-030-92551-2\_11},
  \DOIprefix\doi{10.1007/978-3-030-92551-2\_11}.
\bibitem[{Dennunzio et~al.(2018)Dennunzio, Dorigatti, Formenti, Manzoni and
  Porreca}]{dorigatti2018}
\bibinfo{author}{Dennunzio, A.}, \bibinfo{author}{Dorigatti, V.},
  \bibinfo{author}{Formenti, E.}, \bibinfo{author}{Manzoni, L.},
  \bibinfo{author}{Porreca, A.E.}, \bibinfo{year}{2018}.
\newblock \bibinfo{title}{Polynomial equations over finite, discrete-time
  dynamical systems}, in: \bibinfo{booktitle}{Proc. of {ACRI}'18}, pp.
  \bibinfo{pages}{298--306}.
\bibitem[{Formenti et~al.(2021)Formenti, R{\'e}gin and Riva}]{riva2021mdd}
\bibinfo{author}{Formenti, E.}, \bibinfo{author}{R{\'e}gin, J.C.},
  \bibinfo{author}{Riva, S.}, \bibinfo{year}{2021}.
\newblock \bibinfo{title}{{MDD}s boost equation solving on discrete dynamical
  systems}, in: \bibinfo{booktitle}{International Conference on Integration of
  Constraint Programming, Artificial Intelligence, and Operations Research},
  \bibinfo{organization}{Springer}. pp. \bibinfo{pages}{196--213}.
\bibitem[{Jongsma(2019)}]{jongsma2019basic}
\bibinfo{author}{Jongsma, C.}, \bibinfo{year}{2019}.
\newblock \bibinfo{title}{Basic set theory and combinatorics}, in:
  \bibinfo{booktitle}{Introduction to Discrete Mathematics via Logic and
  Proof}. \bibinfo{publisher}{Springer}, pp. \bibinfo{pages}{205--253}.
\bibitem[{Mara{\~{n}}{\'{o}}n et~al.(2008)Mara{\~{n}}{\'{o}}n, Encinas and del
  Rey}]{AlvarezER08}
\bibinfo{author}{Mara{\~{n}}{\'{o}}n, G.{\'{A}}.}, \bibinfo{author}{Encinas,
  L.H.}, \bibinfo{author}{del Rey, {\'{A}}.M.}, \bibinfo{year}{2008}.
\newblock \bibinfo{title}{A multisecret sharing scheme for color images based
  on cellular automata}.
\newblock \bibinfo{journal}{Information Sciences} \bibinfo{volume}{178},
  \bibinfo{pages}{4382--4395}.
\bibitem[{Nandi et~al.(1994)Nandi, Kar and Chaudhuri}]{NandiKC94}
\bibinfo{author}{Nandi, S.}, \bibinfo{author}{Kar, B.K.},
  \bibinfo{author}{Chaudhuri, P.P.}, \bibinfo{year}{1994}.
\newblock \bibinfo{title}{Theory and applications of cellular automata in
  cryptography}.
\newblock \bibinfo{journal}{{IEEE} Trans. Computers} \bibinfo{volume}{43},
  \bibinfo{pages}{1346--1357}.
\bibitem[{Perez and Régin(2015)}]{perez2015efficient}
\bibinfo{author}{Perez, G.}, \bibinfo{author}{Régin, J.C.},
  \bibinfo{year}{2015}.
\newblock \bibinfo{title}{Efficient operations on {MDD}s for building
  constraint programming models}, in: \bibinfo{booktitle}{IJCAI 2015}, pp.
  \bibinfo{pages}{374--380}.
\bibitem[{Sen{\'{e}}(2012)}]{Sene12}
\bibinfo{author}{Sen{\'{e}}, S.}, \bibinfo{year}{2012}.
\newblock \bibinfo{title}{On the bioinformatics of automata networks}.
\newblock \bibinfo{type}{{HDR}}. University of \'{E}vry Val d'Essonne, France.
\newblock \URLprefix \url{https://tel.archives-ouvertes.fr/tel-00759287}.
\bibitem[{Siebert(2009)}]{Siebert2009}
\bibinfo{author}{Siebert, H.}, \bibinfo{year}{2009}.
\newblock \bibinfo{title}{Dynamical and structural modularity of discrete
  regulatory networks}, in: \bibinfo{booktitle}{{COMPMOD}}, pp.
  \bibinfo{pages}{109--124}.

\end{thebibliography}

\end{document}